\documentclass[11pt]{amsart}

\font\tengoth=eufm10
\font\sevengoth=eufm7
\font\fivegoth=eufm5
\newfam\gothfam  \textfont\gothfam=\tengoth
\scriptfont\gothfam=\sevengoth
\scriptscriptfont\gothfam=\fivegoth
\usepackage{amssymb, amscd}
\usepackage{amsmath}
\usepackage{enumerate}

\newcommand\lieg{{\mathfrak g}}
\newcommand\liek{{\mathfrak k}}
\newcommand\liel{{\mathfrak l}}
\newcommand\lieh{{\mathfrak h}}
\newcommand\liea{{\mathfrak a}}
\newcommand\lien{{\mathfrak n}}
\newcommand\liet{{\mathfrak t}}
\newcommand\liem{{\mathfrak m}}
\newcommand\lieb{{\mathfrak b}}
\newcommand\liep{{\mathfrak p}}
\newcommand\lieq{{\mathfrak q}}
\newcommand\lier{{\mathfrak r}}
\newcommand\lies{{\mathfrak s}}

\newcommand\liey{{\mathfrak y}}

\newcommand\liegl{\mathfrak {gl}}
\newcommand\liesl{\mathfrak {sl}}
\newcommand\liesp{\mathfrak {sp}}
\newcommand\lieu{{\mathfrak u}}
\newcommand\lieso{\mathfrak {so}}

\newcommand\lief{\mathfrak {f}}
\newcommand\CC{\mathbb C}

\newcommand\RR{\mathbb R}
\newcommand\ZZ{\mathbb Z}
\newcommand\NN{\mathbb N}

\newcommand\SO{{\mathrm{SO}}}
\newcommand\SU{{\mathrm{SU}}}
\newcommand\U{{\mathrm{U}}}
\newcommand\calU{{\mathcal{U}}}

\newcommand\calS{{\mathcal{S}}}
\newcommand\calB{{\mathcal{B}}}

\makeatletter 
\newcommand\matc[4]{\left( {#1\@@atop #3}{#2\@@atop #4}\right)}

\newcommand\matr[4]{\left( {\hfill #1\@@atop\hfill #3}{\hfill
#2\@@atop\hfill #4}\right)}

\newcommand\matl[4]{\left( { #1\@@atop #3}{ #2\@@atop\hfill #4}\right)}
\makeatother

\theoremstyle{plain}
\newtheorem{thm}{Theorem}[section]
\newtheorem{lem}[thm]{Lemma}
\newtheorem{prop}[thm]{Proposition}
\newtheorem{cor}[thm]{Corollary}
\theoremstyle{definition}
\newtheorem{defn}[thm]{Definition}
\theoremstyle{remark}

\title[The classifying ring]{\sc The image of the Lepowsky homomorphism for
the group $F_4$}
\author{A. Brega}
\address{CIEM-FaMAF, Universidad Nacional de
C\'or\-do\-ba, C\'or\-do\-ba~5000, Argentina}
\email{brega@mate.uncor.edu}
\author{L. Cagliero}
\address{CIEM-FaMAF, Universidad Nacional de C\'or\-do\-ba,
C\'or\-do\-ba~5000, Argentina} \email{cagliero@mate.uncor.edu}
\author{J. Tirao}
\address{CIEM-FaMAF, Universidad Nacional de C\'or\-do\-ba,
C\'or\-do\-ba~5000, Argentina} \email{tirao@mate.uncor.edu}

\begin{document}
\begin{abstract}
Let $G_o$ be a semisimple Lie group, let $K_o$ be a maximal compact
subgroup of $G_o$ and let $\liek\subset\lieg$ denote the complexification
of  their Lie algebras. Let $G$ be the adjoint group of $\lieg$ and
let $K$ be the connected Lie subgroup of $G$ with Lie algebra $ad(\liek)$.
If $U(\lieg)$ is the universal enveloping algebra of $\lieg$ then
$U(\lieg)^K$ will denote the centralizer of $K$ in $U(\lieg)$.
Also let $P:U(\lieg)\longrightarrow U(\liek)\otimes U(\liea)$ be the
projection map corresponding to the direct sum
$U(\lieg)=\bigl(U(\liek)\otimes U(\liea)\bigr)\oplus U(\lieg)\lien$
associated to an Iwasawa decomposition of $G_o$ adapted to $K_o$. In
this paper we give a characterization of the image of $U(\lieg)^K$
under the injective antihomorphism $P:U(\lieg)^K\longrightarrow
U(\liek)^M\otimes U(\liea)$, considered by Lepowsky in \cite{L}, when
$G_o$ is locally isomorphic to F$_4$.
\end{abstract}

\smallskip

\maketitle
\section{Introduction}\label{Intr}

Let $G_o$ be a connected, noncompact, real semisimple Lie group with
finite center, and let $K_o$ denote a maximal compact subgroup of
$G_o$.  We denote with $\lieg_o$ and $\liek_o$ the Lie algebras of
$G_o$ and $K_o$, and  $\liek\subset\lieg$ will denote the respective
complexified Lie algebras. Let $G$ be the adjoint group of $\lieg$ and
let $K$ be the connected Lie subgroup of $G$ with Lie algebra $ad(\liek)$.
Let $U(\lieg)$ be the universal enveloping algebra of $\lieg$ and let
$U(\lieg)^K$ denote the centralizer of $K$ in $U(\lieg)$.

In order to contribute to the understanding of $U(\lieg)^K$ Kostant
suggested to consider the projection map $P:U(\lieg)\longrightarrow
U(\liek)\otimes U(\liea)$, corresponding to the direct sum
$U(\lieg)=\bigl(U(\liek)\otimes U(\liea)\bigr)\oplus U(\lieg)\lien$
associated to an Iwasawa decomposition
$\lieg=\liek\oplus\liea\oplus\lien$ adapted to $\liek$. In \cite{L}
Lepowsky studied the restriction of $P$ to $U(\lieg)^K$ and proved,
among other things, that one has the following exact sequence
$$0\longrightarrow U(\lieg)^K\overset P\longrightarrow
U(\liek)^M\otimes U(\liea),$$ where $U(\liek)^M$ denotes the
centralizer of $M$ in $U(\liek)$, $M$  being the centralizer of
$\liea$ in $K$. Moreover if $U(\liek)^M\otimes U(\liea)$ is given
the tensor product algebra structure then $P$ becomes an
antihomomorphism of algebras. Hence to go any further in this
direction it is necessary to determine the image of $P$.

To determine the actual image $P(U(\lieg)^K)$ Tirao introduced in \cite{Ti}
a subalgebra $B$ of $U(\liek)^M\otimes U(\liea)$ defined by a set of
linear equations in $U(\liek)$ derived from certain embeddings between Verma modules
(see \eqref{eqB} below). Then he proved that $P(U(\lieg)^K)$ always lies in
$B$, and furthermore that $P(U(\lieg)^K)=B^{W_\rho}$ when $G_o$ = SO($n$,1) or SU($n$,1).
Here $W$ is the Weyl group of the pair $(\lieg,\liea)$, $\rho$ is half
the sum of the positive roots of $\lieg$ and $B^{W_\rho}$ is the subalgebra
of all elements in $B$ that are invariant under the tensor product action
of $W$ on $U(\liek)^M$ and the translated action of $W$ on $U(\liea)$
(see Corollary 3.3 of \cite{KT}).

In \cite{BCT} we proved that $P(U(\lieg)^K)=B$ when $G_0=$ Sp$(n,1)$,
and more recently, we showed that $B^{W_\rho}=B$ when $G_0=$ SO($n$,1) or
SU($n$,1) (see \cite{BCT1}). Hence these results established that $P(U(\lieg)^K)=B$
for every classical real rank one semisimple Lie group with finite center. This
paper is devoted to proving that this result also holds for the exceptional
Lie group F$_4$. The main result of the present work is the following,
\begin{thm}\label{main}
If $G_o$ is locally isomorphic to F$_4$, then $P(U(\lieg)^K) =B$.
\end{thm}
\noindent This result confirms our old belief that the following theorem holds,
\begin{thm}\label{main1}
Let $G_o$ be a real rank one  semisimple Lie group. Then the image of the Lepowsky
homomorphism $P$ is the algebra $B$.
\end{thm}
We point out that our proof of Theorem \ref{main1} follows a general pattern in all cases, however at certain points in the argument there are some substantial differences. Certainly the cases of Sp($n$,1) and
F$_4$ are the most difficult of the rank one groups.

\smallskip

The proof of Theorem \ref{main} follows the same ideas used to prove the analogue
theorem for the symplectic group Sp$(n,1)$, however we had to overcome some difficulties
to establish the transversality results needed (Section \ref{trans.}), and the a priori
estimates of the Kostant degrees (Section \ref{Kostant}). In Section \ref{trans.} we give
a new and simplified version of the corresponding transversality results obtained in the
symplectic case (see Section 4 of \cite{BCT}). This new version
is sufficient because of the introduction of a simplifying hypothesis called the
\textit{degree property}, which is done in Section \ref{Kostant}. In this section we
use this property to obtain an a priori estimate of the Kostant degree
of certain elements $b\in B$. This allows us to reduce the proof of Theorem \ref{main}
to proving Theorem \ref{Btilde2} (see Proposition \ref{Btilde4}). The proof of this last theorem  is given in Section \ref{mainproof}, following the ideas
developed in the symplectic case. In fact, most of the results proved in Section 6 of
\cite{BCT} hold in this case with appropriate changes.

\section{The algebra B and the image of $U(\lieg)^K$}\label{algebraB}

In this section we assume that $G_o$ be a connected, noncompact, real
semisimple Lie group with finite center and of split rank one, not locally
isomorphic to SL$(2,\RR)$. Let $\liet_o$ be a Cartan subalgebra of the
Lie algebra $\liem_o$ of $M_o$.
Set $\lieh_o=\liet_o\oplus\liea_o$ and let $\lieh=\liet\oplus\liea$ be the
corresponding complexification. Then $\lieh_o$ and $\lieh$ are
Cartan subalgebras of $\lieg_o$ and $\lieg$, respectively. Choose a
Borel subalgebra $\liet\oplus\liem^+$ of the complexification
$\liem$ of $\liem_o$ and take $\lieb=\lieh\oplus\liem^+\oplus\lien$
as a Borel subalgebra of $\lieg$. Let $\Delta$ and $\Delta^+$ be,
respectively, the corresponding sets of roots and positive roots of
$\lieg$ with respect to $\lieh$. If $\alpha\in\Delta$ let $X_\alpha$
denote a nonzero root vector associated to $\alpha$. Also, let $\theta$ be
be the Cartan involution, and let $\lieg=\liek\oplus\liep$ be the
Cartan decomposition of $\lieg$ corresponding to $(G_o,K_o)$.

\smallskip

If $\langle\,,\,\rangle$ denotes the Killing form of $\lieg$, for
each $\alpha\in\Delta$ let $H_\alpha\in\lieh$ be the unique element
such that $\phi(H_\alpha)=2\langle\phi,\alpha\rangle/
\langle\alpha,\alpha\rangle$ for all $\phi\in\lieh^*$, and let
$\lieh_\RR$ be the real span of $\{H_\alpha:\alpha\in \Delta\}$.
Also set $H_\alpha=Y_\alpha+Z_\alpha$ where $Y_\alpha\in\liet$ and
$Z_\alpha\in\liea$, and let $P_+=\{\alpha\in\Delta^+:Z_\alpha\neq0\}$.
If $\alpha\in P_+$ is a simple root and $\lambda=\alpha|_\liea$ we
let $\lieg(\lambda)$ denote the real reductive rank one subalgebra of
$\lieg_o$ associated to $\lambda$. We point out that $Y_\alpha\ne0$
if and only if $[\lieg(\lambda),\lieg(\lambda)]\not\simeq
\liesl(2,\RR)$. Hence, by our assumptions on $G_o$, for any simple root
$\alpha\in P_+$ we have $Y_\alpha\neq0$.

\smallskip

If $\alpha\in P_+$ we have $\liea={\CC}Z_\alpha$ and we can view the
elements in $U(\liek)\otimes U(\liea)$ as polynomials in $Z_\alpha$ with
coefficients in $U(\liek)$. For any such a root $\alpha$ we set
$E_\alpha=X_{-\alpha}+\theta X_{-\alpha}$. Now we introduce the subalgebra $B$
of $U(\liek)^M\otimes U(\liea)$ defined by Tirao in \cite{Ti}.

\begin{defn}\label{B}
Let $B$ be the algebra of all $b\in U(\liek)^M\otimes U(\liea)$
that satisfy
\begin{equation}\label{eqB}
E_\alpha^nb(n-Y_\alpha-1)\equiv b(-n-Y_\alpha-1)E_\alpha^n,
\end{equation}
for all simple roots $\alpha\in P_+$ and all $n\in\NN$. Here the congruence
is module the left ideal $U(\liek){\liem}^+$ of $U(\liek)$. This definition is
motivated by equations \eqref{eq1} (see Corollary \ref{cor6}).
\end{defn}
\noindent In Theorem 5 of \cite{Ti} Tirao proved
that $B$ is a subalgebra of $U(\liek)^M\otimes U(\liea)$, and in
Corollary 6 of the same paper he proved
that $P(U(\lieg)^K)\subset B$. For further reference we state
this Corollary.

\begin{cor}\label{cor6}
Let $\alpha\in P_+$ be a
simple root. Then for all $n\in\NN$
and $u\in U(\lieg)^K$ we have
\begin{equation}\label{eq1}
E_\alpha^nP(u)(n-Y_\alpha-1)\equiv P(u)(-n-Y_\alpha-1)E_\alpha^n,
\end{equation}
where the congruence is module $U(\liek){\liem}^+$.
\end{cor}

\smallskip

In order to prove Theorem \ref{main} we will now introduce some notation and
recall known results. Let $\Gamma$ denote the set of all equivalence classes of
irreducible holomorphic finite dimensional $K$-modules $V_\gamma$
such that $V_\gamma^M\not=0$. Any $\gamma\in\Gamma$ can be realized
as a submodule of all harmonic polynomial functions on $\liep$,
homogeneous of degree $d$, for a uniquely determined $d=d(\gamma)$
(see \cite{KR}). We shall refer to the non negative integer $d(\gamma)$ as
the {\em Kostant degree} of $\gamma$. If $V$ is any $K$-module and
$\gamma\in\hat K$ then $V_\gamma$ will denote the isotypic component
of $V$ corresponding to $\gamma$. Let $V$ be a locally finite $K$-module
such that $V^M\neq 0$ and let $v\in V^M$, $v\neq 0$. Since $V$ is locally
finite, we can decompose $v$ into $K$-isotypic $M$-invariants as follows
$$v= \sum_{\gamma \in \Gamma}\textbf{v}_{\gamma},$$
where $\textbf{v}_{\gamma}\in V_\gamma$ denotes the $\gamma$-isotypic
component of $v$. Then we define the {\emph{Kostant degree}} of $v$ by,
\begin{equation}\label{Kostdeg}
d(v) = \max \{d(\gamma): \textbf{v}_{\gamma}\neq 0\}.
\end{equation}

\noindent Since we are mainly concerned with representations $\gamma\in\Gamma$
that occur as subrepresentations of $U(\liek)$ we set,
\begin{equation}\label{Gamma1}
\Gamma_1=\{\gamma\in \Gamma:\gamma \text{ is a subrepresentation of
} U(\liek)\}.
\end{equation}

\smallskip

If $0\ne b\in U(\liek)\otimes U(\liea)$ we can write $b=b_m\otimes
Z^m+\cdots+b_0$ in a unique way with $b_j\in U(\liek)$ for
$j=0,\dots,m$, $b_m\ne0$ and $Z=Z_\alpha$ for any $\alpha\in P_+$
simple. We shall refer to $b_m$ (resp. $\widetilde b=b_m\otimes
Z^m$) as the {\em leading coefficient} (resp. {\em leading term}) of
$b$ and to $m$ as the {\em degree} of $b$. Also, let 0 be the
leading coefficient and the leading term of $b=0$.

Let $M'_o$ be the normalizer of $A_o$ in $K_o$ and let $W=M'_o/M_o$
be the Weyl group of $(G_o,K_o)$. Then $\big(U(\liek)^M\otimes
U(\liea)\big)^W$ denotes the ring of $W$-invariants in
$U(\liek)^M\otimes U(\liea)$, under the tensor product action of the
natural actions of $W$ on $U(\liek)^M$ and $U(\liea)$, respectively.

\smallskip

At this point it is convenient to state the following result. Its
proof is given in Proposition 2.6 of \cite{BCT}, using the techniques
and the notation of Section 3 of \cite{Ti}.

\begin{prop}\label{cor9} If $\widetilde b=b_m\otimes Z^m\in
\big(U(\liek)^M\otimes U(\liea)\big)^W$ and $d(b_m)\le m$, then
there exits $u\in U(\lieg)^K$ such that $\widetilde b$ is the
leading term of $b=P(u)$.
\end{prop}

\smallskip

Last proposition suggests using an inductive argument to prove
Theorem \ref{main}. To do this it is sufficient to establish
the following theorem. In fact, in Proposition \ref{main''} below
we prove that Theorem \ref{main'} implies Theorem \ref{main}.

\begin{thm}\label{main'} If $b=b_m\otimes Z^m+\cdots+b_0\in
B$ and $b_m \ne 0$, then $d(b_m)\le m$ and its leading term
$b_m\otimes Z^m\in\big(U(\liek)^M\otimes U(\liea)\big)^W$.
\end{thm}

\noindent {\bf{Remark:}} In F$_4$ the non trivial element of $W$ can be
represented by an element in $M'_o$ which acts on $\lieg$ as the
Cartan involution. Hence, to prove that the leading term $b_m\otimes
Z^m$ is $W$-invariant it is enough to show that $m$ is even.

\begin{prop}\label{main''}
Theorem \ref{main'} implies Theorem \ref{main}.
\end{prop}
\begin{proof}Assume that Theorem \ref{main'} holds. From
Corollary \ref{cor6} we know that $P(U(\lieg)^K)\subset B$.
Then let us prove by induction on the degree $m$ of $b\in B$, that
$B\subset P(U(\lieg)^K)$. If $m=0$ we have $b=b_0\in U(\liek)^M$ and
Theorem \ref{main'} implies that $d(b_0)=0$.
If $\gamma\in\Gamma$ and $d(\gamma)=0$ then $\gamma$ can be realized by
constant polynomial functions on $\liep$ and these fuctions are
$K$-invariant.
Thus $b_0\in U(\liek)^K$ and therefore $b=b_0=P(b_0)\in P(U(\lieg)^K)$.

If $b\in B$ and $m>0$, from
Theorem \ref{main'} and Proposition \ref{cor9} we know that there
exists $v\in U(\lieg)^K$ such that $\widetilde{P(v)}=\widetilde b$.
Then $b-P(v)$ lies in $B$ and the degree of $b-P(v)$ is strictly
less than $m$.
Hence, by the induction hypothesis, there exists $u\in U(\lieg)^K$ such
that $P(u)=b-P(v)$ and  $b=P(u+v)\in P(U(\lieg)^K)$.
This completes the induction argument and we obtain that
$B\subset P(U(\lieg)^K)$, as we wanted to prove.\qed
\end{proof}

\smallskip

In view of this last result the main objective of this paper is to
prove Theorem \ref{main'} when $G_0$ is locally isomorphic to F$_4$.

\section{The equations defining $B$}\label{sec.eqE}

From now on we shall write $u\equiv v$ instead of $u\equiv v \mod
\big(U(\liek)\liem^+\big)$, for any $u,v\in U(\liek)$. Next
result was proved in Lemma 29 of \cite{Ti} for $G_o$ of arbitrary rank.

\begin{lem}\label{defc} Let $\alpha\in P_+$ be a simple root. Set
$H_\alpha=Y_\alpha+Z_\alpha$ where $Y_\alpha\in\liet$,\, $Z_\alpha\in\liea$
and let $c=\alpha(Y_\alpha)$. If $\lambda=\alpha|_\liea$ and
$m(\lambda)$ is the multiplicity of $\lambda$, then $c=1$ when
$2\lambda$ is not a restricted root and $m(\lambda)$ is even, or
when $m(\lambda)$ is odd, and $c=\frac32$ when $2\lambda$ is a
restricted root and $m(\lambda)$ is even.
\end{lem}

In particular, if $G_o$ is locally isomorphic to
F$_4$ we have $c=\frac32$. To simplify the notation set $E=E_\alpha$,
$Y=Y_\alpha$ and $Z=Z_\alpha$ for any simple root $\alpha\in P_+$.
Notice that $[E,Y]=cE$, where $c$ is as
in Lemma \ref{defc}. Also, since $E_\alpha=X_{-\alpha}+\theta X_{-\alpha}$
and $\alpha$ is a simple root in $P_+$ it follows that $E$ is $\liem^+$- dominant.

\smallskip

We shall identify $U(\liek)\otimes U(\liea)$ with the polynomial ring in
one variable $U(\liek)[x]$, replacing $Z$ by the
indeterminate $x$. To study equation \eqref{eqB} we change
$b(x)\in U(\liek)[x]$ by $c(x)\in U(\liek)[x]$ defined by
\begin{equation}\label{c2}
c(x)=b(x+H-1),
\end{equation}
where $H$ is an appropriate vector in $\liet$ to be chosen later,
depending on the simple root $\alpha\in P_+$ and such that
$[H,E]=\frac12 E$ (see \eqref{H}). Now, if $\widetilde Y=Y+H$ we have
$[E,\widetilde Y]=E$. Then $b(x)\in U(\liek)[x]$ satisfies \eqref{eqB}
if and only if $c(x)\in U(\liek)[x]$ satisfies
\begin{equation}\label{ec5}
E^nc(n-\widetilde Y)\equiv c(-n-\widetilde Y)E^n
\end{equation}
\noindent for all $n\in\NN$. Note that \eqref{ec5} is an
equation in the noncommutative ring $U(\liek)$.

\smallskip

Now, if $p$ is a polynomial in one indeterminate $x$ with
coefficients in a ring let $p^{(n)}$ denote the n-th discrete
derivative of $p$. That is, $p^{(n)}(x)=\sum^n_{j=0} (-1)^j {\binom n
j}p(x+{\frac n2}-j)$. In particular, if $p=p_mx^m+\cdots+p_0$ we have
\begin{equation*}
p^{(n)}(x)=\begin{cases} 0,&\text{if $n>m$}\\ m! p_m ,&\text{if
$n=m$}.
\end{cases}
\end{equation*}

Also, if $X\in\liek$ we shall denote with $\dot X$ the derivation of
$U(\liek)$ induced by ad($X$). Moreover if $D$ is a derivation of
$U(\liek)$ we shall denote with the same symbol the unique
derivation of $U(\liek)[x]$ which extends $D$ and such that $Dx=0$.
Thus for $b\in U(\liek)[x]$ and $b=b_mx^m+\cdots+b_0$, we have
$Db=(Db_m)x^m+\cdots+(Db_0)$. Observe that these derivations commute
with the operation of taking the discrete derivative in
$U(\liek)[x]$.

\smallskip

Next theorem gives a triangularized version of the system
(\ref{ec5}), and in turn, of the system (\ref{eqB}) that defines the
algebra $B$. A proof of it is given in \cite{BC}, where the system
\eqref{ec5} is studied in a more abstract setting and in particular the
LU-decomposition of its coefficient matrix is given.

\begin{thm}\label{thm1} Let $c\in U(\liek)[x]$. Then the following
systems of equations are equivalent:

\noindent(i) $E^n c(n-\widetilde Y)\equiv c(-n-\widetilde Y)E^n$,
$(n\in\NN_0)$;

\smallskip

\noindent (ii) $\dot E^{n+1}\big(c^{(n)}\big)(\frac n2+1-\widetilde
Y)+\dot E^n\big(c^{(n+1)}\big)(\frac n2-\frac12-\widetilde Y)E\equiv
0$, $(n\in\NN_0)$.

\smallskip

\noindent Moreover, if $c\in U(\liek)[x]$ is a solution of one of the above
systems, then for all $\ell,n\in\NN_0$ we have

\smallskip

\noindent (iii) $(-1)^n\dot E^{\ell}\big(c^{(n)}\big)(-\frac
n2+\ell-\widetilde Y)E^n-(-1)^\ell\dot
E^n\big(c^{(\ell)}\big)(-\frac \ell2+n-\widetilde Y)E^\ell\equiv 0$.
\end{thm}

\smallskip

Observe that if $c\in U(\liek)[x]$ is of degree $m$ and
$c=c_mx^m+\cdots+c_0$, then all the equations of the system (ii)
corresponding to $n>m$ are trivial, because $c^{(n)}=0$. Moreover
the equation corresponding to $n=m$ reduces to $\dot
E^{m+1}(c_m)\equiv0$, and more generally the equation associated to
$n=j$ only involves the coefficients $c_m,\dots,c_j$. In this sense
the system (ii) is a triangular system of $m+1$ linear equations in
the $m+1$ unknowns $c_m,\dots,c_0$.

\medskip

If $0\neq b(x)\in U(\liek)[x]$ and $c(x)\in U(\liek)[x]$ is defined
by $c(x)=b(x+H-1)$, where $H$ is as in \eqref{H}, we find it
convenient to write, in a unique way, $b=\sum_{j=0}^m b_jx^j$ with
$b_j\in U(\liek)$, $b_m\ne0$, and $c=\sum_{j=0}^m c_j\varphi_j$
where $c_j\in U(\liek)$ and $\{\varphi_n\}_{n\ge0}$ is the basis of
$\CC[x]$ defined by,
\begin{alignat}{2}
\varphi_0&=1, &&\qquad\qquad\ \tag{i}\\
\varphi_n^{(1)}&=\varphi_{n-1} && \text{\qquad if } n\geq1,\tag{ii}\\
\varphi_n(0)&=0  &&\text{\qquad if }n\ge1.\tag{iii}
\end{alignat}
Moreover it is easy to prove that such a family is given by
\begin{equation}\label{phi}
\varphi_n(x)={\textstyle\frac1{n!}x(x+\frac n2-1)(x+\frac
n2-2)\cdots(x-\frac n2+1)},\quad n\ge1.
\end{equation}
Next lemma contains the results of Lemma 3.3 and Lemma 3.5
of \cite{BCT}. Its proof is the same as that of the corresponding
lemmas in \cite{BCT}.
\begin{lem}\label{cb} Let $b=\sum_{j=0}^mb_jx^j\in U(\liek)[x]$ and set
$c(x)=b(x+H-1)$. Then, if $c=\sum_{j=0}^m c_j\varphi_j$ with $c_j\in
U(\liek)$ we have
\begin{equation*}\label{cb1}
c_i=\sum^m_{j=i}b_jt_{ij}\qquad 0\le i\le m,
\end{equation*}
where
\begin{equation*}\label{tij}
t_{ij}=\sum_{k=0}^i(-1)^k\binom ik(H+{\textstyle\frac i2}-1-k)^j.
\end{equation*}
Thus $t_{ij}$ is a polynomial in $H$ of degree $j-i$. Moreover,
$$\dot E^{j-i}(t_{ij})=\left({-\frac12}\right)^{j-i}j!E^{j-i}.$$
\end{lem}

From these results and Theorem \ref{thm1} we obtain the following
theorem and its corollary in the same way as in \cite{BCT}.

\begin{thm}\label{Ec} If $b=b_m\otimes Z^m+\cdots+b_0\in B$, then $\dot
E^{m+1}(c_j)\equiv0$ for all $0\le j\le m$.
\end{thm}

\begin{cor}\label{Eb} If $b=b_m\otimes Z^m+\cdots+b_0\in B$, then $\dot
E^{2m+1-j}(b_j)\equiv0$ for all $0\le j\le m$.
\end{cor}

Next we rewrite equation (iii) of Theorem \ref{thm1} for later
reference. Given $b=\sum_{j=0}^mb_jx^j\in B$ and $c(x)=b(x+H-1)$ as
above, it follows from Theorem \ref{Ec} that equation (iii) of
Theorem \ref{thm1} is satisfied if $\ell>m$ or $n>m$, and it is
trivial when $\ell=n$. Also note that the equation corresponding to
$(n,\ell)$ is  equivalent to that one corresponding to $(\ell,n)$.

\begin{thm}\label{thm2} Let $b=\sum_{j=0}^mb_jx^j\in U(\liek)[x]$ and
$c(x)=b(x+H-1)$. If $c=\sum_{j=0}^m c_j\varphi_j$ with $c_j\in
U(\liek)$ and $0\le\ell,n$ we set
\begin{equation*}\label{eln}
\begin{split}
\epsilon(\ell,n)&=(-1)^n\sum_{n\le i\le m}\dot
E^{\ell}(c_i)\varphi_{i-n}
(-{\textstyle\frac n2}+\ell-\widetilde Y)E^n\\
&-(-1)^\ell\sum_{\ell\le i\le m}\dot E^n(c_i)\varphi_{i-\ell}
(-{\textstyle\frac\ell2}+n-\widetilde Y)E^\ell.
\end{split}
\end{equation*}
Then, if $b\in B$ we have
$\epsilon(\ell,n)\equiv0\mod(U(\liek)\liem^+)$ for all $0\le\ell,n$.
\end{thm}
\begin{proof} The assertion follows from equation (iii) of Theorem \ref{thm1}
and the fact that $c^{(k)}=\sum_{i=k}^m c_i\varphi_{i-k}$ for all
$0\le k\le m$.
\end{proof}

\section{The group F$_4$}\label{F4}

Let $G_o$ be locally isomorphic to F$_4$. The Dynkin-Satake diagram of $\lieg$,
the complexification of the Lie algebra of $G_o$, is

\noindent

\

\setlength{\unitlength}{.75mm}
\begin{picture}(0,15)
\put(59,14){$\bullet$}
\put(62,15.5){\line(1,0){8}} \put(58,10){\scriptsize$\alpha_4$}
\put(71,14){$\bullet$}
\put(74,15.85){\line(1,0){6}}
\put(74,15.15){\line(1,0){6}}
\put(79,14){$>$} \put(70,10){\scriptsize$\alpha_3$}
\put(83,14){$\bullet$}
\put(86,15.5){\line(1,0){8}} \put(82,10){\scriptsize$\alpha_2$}
\put(95,14){$\circ$}
\put(94,10){\scriptsize$\alpha_1$}
\end{picture}

\vspace{-.7cm}

\noindent We can choose an orthonormal basis
$\{\epsilon_i\}_{i=1}^{4}$ of $\lieh_\RR^*$  such that
$\alpha_4=\epsilon_2-\epsilon_3$, $\alpha_3=\epsilon_3-\epsilon_4$,
$\alpha_2=\epsilon_4$, $\alpha_1=\frac12(\epsilon_1-\epsilon_2-\epsilon_3-\epsilon_4)$.
Moreover, if $\sigma$ denotes the conjugation of $\lieg$ with respect to $\lieg_o$,
then $\epsilon_1^\sigma=\epsilon_1$ and $\epsilon_i^\sigma=-\epsilon_i$ if $2\le i\le 4$.
Also, we have $\epsilon_1^\theta=-\epsilon_1$ and
$\epsilon_i^\theta=\epsilon_i$ for $2\le i\le 4$. From the diagram
it follows that
\begin{align*}
\begin{split}
\Delta^+(\lieg,\lieh)=&\{\epsilon_i:1\le
i\le4\}\cup\{\epsilon_i\pm\epsilon_j:1\le i<j\le4\}\\
&\cup\{\textstyle\frac12(\epsilon_1\pm\epsilon_2\pm\epsilon_3\pm\epsilon_4)\},\\
P_+=&\{\epsilon_1,\epsilon_1\pm\epsilon_2,\epsilon_1\pm\epsilon_3,\epsilon_1
\pm\epsilon_4\}\cup
\{\textstyle\frac12(\epsilon_1\pm\epsilon_2\pm\epsilon_3\pm\epsilon_4)\},\\
P_-=&\{\epsilon_2,\epsilon_3,\epsilon_4,\epsilon_2\pm\epsilon_3,\epsilon_2\pm
\epsilon_4,\epsilon_3\pm\epsilon_4\},
\end{split}
\end{align*}
where the signs may be chosen independently. Here $P_-$ denotes the
set of roots in $\Delta^+(\lieg,\lieh)$ that vanish on $\liea$.
Hence $P_-=\Delta^+(\liem,\liet)$  and from this it follows that
$\liem\simeq\lieso(7,\CC)$.

We have $\liet=\ker(\epsilon_1)$ and $\epsilon_1$ is the only root in $P_+$
that vanishes on $\liet$. If we set $\mu=\epsilon_1$, then $H_\mu=Z_\mu\in\liea$.
Choose the root vector $X_\mu$ so that $\langle X_\mu,\theta X_\mu\rangle=2$ and define
$X_{-\mu}=\theta X_\mu$. Then the
ordered set $\{H_\mu,X_\mu,X_{-\mu}\}$ is an s-triple. This choice
characterizes $X_\mu$ up to a sign. On the other hand, it can be established that
for any choice of nonzero root vectors $X_{\alpha_1}$ and $X_{-\alpha_1}$ we have
$[X_\mu,\theta X_{\alpha_1}]=tX_{\alpha_1}$ and
$[X_\mu,X_{-\alpha_1}]=-t\theta X_{-\alpha_1}$ with $t^2=1$.
Then normalize $X_\mu$ so that,
\begin{equation}\label{Xmu}
[X_\mu,\theta X_{\alpha_1}]=-X_{\alpha_1} \quad {\text{and}} \quad
[X_\mu,X_{-\alpha_1}]=\theta X_{-\alpha_1}.
\end{equation}
Now consider the Cayley transform $\chi$ of $\lieg$ defined by
\begin{equation*}\label{xi}
\chi=Ad(\exp\tfrac\pi4(\theta X_\mu-X_\mu)).
\end{equation*}
It is easy to see that
\begin{equation*}\label{xi1}
Ad(\exp t(\theta
X_\mu-X_\mu))H_\mu=\cos(2t)H_\mu+\sin(2t)(X_\mu+\theta X_\mu).
\end{equation*}
Then $\chi(H_\mu)=X_\mu+\theta X_\mu$ and, since $\mu_{|_\liet}=0$,
$\chi$ fixes all elements of $\liet$. Therefore
$\lieh_\liek=\chi(\liet\oplus\liea)=\liet\oplus\CC(X_\mu+\theta
X_\mu)\subset\liek$ is a Cartan subalgebra of both $\lieg$ and
$\liek$.

For any $\phi\in\lieh^*$ define $\widetilde\phi\in\lieh_\liek^*$ by
$\widetilde\phi=\phi\cdot\chi^{-1}$. Then
$\Delta(\lieg,\lieh_\liek)=\{\widetilde\alpha:\alpha\in\Delta(\lieg,\lieh)\}$
and $\lieg_{\widetilde\alpha}=\chi(\lieg_\alpha)$. A root
$\widetilde\alpha\in\Delta(\lieg,\lieh_\liek)$ is said to be compact
(respectively noncompact) if $\lieg_{\widetilde\alpha}\subset\liek$
(respectively $\lieg_{\widetilde\alpha}\subset\liep$). Let
$\Delta(\liek,\lieh_\liek)$ and $\Delta(\liep,\lieh_\liek)$ denote, respectively,
the sets of compact and noncompact roots.

Using Lemma 3.1 of \cite{BCT} it follows that $\widetilde{\alpha_3}$ and
$\widetilde{\alpha_4}$ are compacts roots, and that $\widetilde{\alpha_2}$ is a
noncompact root. Also, since $X_\mu$ was chosen so that \eqref{Xmu} holds,
we obtain that $\widetilde{\alpha_1}$ is a noncompact root. From this it follows that

\begin{align*}
\begin{split}
\Delta(\liek,\lieh_\liek)=&\{\pm(\widetilde\epsilon_i\pm\widetilde\epsilon_j):1\le i<j\le4\}\\
&\cup\{\textstyle\frac12(\pm\widetilde\epsilon_1\pm\widetilde\epsilon_2
\pm\widetilde\epsilon_3\pm\widetilde\epsilon_4): \text{ even number of minus signs}\},\\
\Delta(\liep,\lieh_\liek)=&\{\pm\widetilde\epsilon_i:1\le  i\le4\}\\
&\cup\{\textstyle\frac12(\pm\widetilde\epsilon_1\pm\widetilde\epsilon_2
\pm\widetilde\epsilon_3\pm\widetilde\epsilon_4): \text{ odd number of minus signs}\}.
\end{split}
\end{align*}

Next we construct a particular Borel subalgebra $b_\liek=\lieh_\liek\oplus\liek^+$
of $\liek$ that will be useful later
on to describe the set $\Gamma$, as well as some of the properties of the elements
of $\Gamma$ (see Proposition \ref{hw3}). For more details on the
construction of the subalgebra $b_\liek$ and its relation with $\Gamma$ we
refer the reader to \cite{CT}.

Since $\alpha_1=\frac12(\epsilon_1-\epsilon_2-\epsilon_3-\epsilon_4)$
is the only simple root in $P_+$ set, as in the previous section,
$E=X_{-\alpha_1}+\theta X_{-\alpha_1}$. Let
$H_+\in\liet_\RR$ be such that $\alpha(H_+)>0$ for all
$\alpha\in\Delta^+(\liem,\liet)$. We say that $H_+$ is  $\liek$-regular
if in addition $\alpha(H_+)\ne0$ for all $\alpha$
with $\widetilde\alpha\in\Delta(\liek,\lieh_\liek)$. Since $\mu$ is the only
root in $\Delta^+(\lieg,\lieh)$ that vanishes on $\liet$ and since
$\widetilde\mu$ is a noncompact root, it follows that $\liek$-regular vectors
exist. Given a $\liek$-regular vector $H_+$ consider the positive system
\begin{equation*}\label{pos.syst}
\Delta^+(\liek,\lieh_\liek)=\{\widetilde\alpha\in\Delta(\liek,\lieh_\liek):
\alpha(H_+)>0\}.
\end{equation*}
If $\lambda_0=\alpha_1|_\liea$ is the simple restricted root and $H_+$
is a $\liek$-regular vector we consider the following set,
$$P_+(\lambda_0)^-=\{\alpha\in P_+:\alpha|_\liea=\lambda_0\text{\;and\;}
\alpha(H_+)<0\}.$$
\begin{defn} A positive system $\Delta^+(\liek,\lieh_\liek)$ defined by a
$\liek$-regular vector $H_+$ (see \eqref{pos.syst}) is said to be compatible
with $E$ if $\alpha-\alpha_1$ is a root for every $\alpha\in P_+(\lambda_0)^-$
such that $\alpha\ne\alpha_1$.
\end{defn}

The $\liek$-regular vectors in $\liet_\RR$, for $\lieg_o\simeq\lief_4$, are all of
the form $H_+=(0,t_2,t_3,t_4)$ with $t_2>t_3>t_4>0$ and $t_2\ne
t_3+t_4$. Different vectors $H_+$ define two different positive
systems, they depend only on whether $\pm(t_2-t_3-t_4)>0$, and they
are both compatible with $E$. From now on fix a $\liek$-regular
vector $H_+=(0,t_2,t_3,t_4)$ with $t_2>t_3>t_4>0$ and $t_2>t_3+t_4$.
The corresponding positive system in $\Delta(\liek,\lieh_\liek)$ is,
\begin{align*}
\begin{split}
\Delta^+(\liek,\lieh_\liek)=&\{\widetilde\epsilon_i\pm\widetilde\epsilon_j:2\le
i<j\le4\}\cup\{\widetilde\epsilon_i\pm\widetilde\epsilon_1:2\le i\le 4\}\\
&\cup\{\textstyle\frac12(\pm\tilde\epsilon_1+
\widetilde\epsilon_2\pm\widetilde\epsilon_3\pm\widetilde\epsilon_4):
\text{even number of minus signs}\},
\end{split}
\end{align*}
and $b_\liek=\lieh_\liek\oplus\liek^+$ is the associated Borel subalgebra.
A simple system in $\Delta^+(\liek,\lieh_\liek)$ is given by,
\begin{equation}\label{simroot}
\Pi(\liek,\lieh_\liek)=\{\widetilde\epsilon_4+\widetilde\epsilon_1, \,
\widetilde\epsilon_3-\widetilde\epsilon_4, \,
\widetilde\epsilon_4-\widetilde\epsilon_1, \,
\tfrac12(\widetilde\epsilon_1+\widetilde\epsilon_2-
\widetilde\epsilon_3-\widetilde\epsilon_4)\}.
\end{equation}
Hence $\liek\simeq\lieso(9,\CC)$.

\smallskip

Fix nonzero root vectors $X_{\epsilon_i+\epsilon_1}$ ($2\le i\le 4$),
$X_{\epsilon_i\pm \epsilon_j}$ ($2\le i<j\le 4$) and define,
\begin{equation}\label{def1}
X_{\widetilde\epsilon_i+\widetilde\epsilon_1}=\chi(X_{\epsilon_i+\epsilon_1}), \quad
X_{\widetilde\epsilon_i-\widetilde\epsilon_1}=\chi(\theta
X_{\epsilon_i+\epsilon_1}), \quad
X_{\widetilde\epsilon_i\pm\widetilde\epsilon_j}=\chi(X_{\epsilon_i\pm\epsilon_j}).
\end{equation}

\noindent Then it follows from Proposition 2.4 of \cite{CT} that,
\begin{equation}\label{ipmj}
X_{\widetilde\epsilon_i\pm \widetilde\epsilon_j}=X_{\epsilon_i\pm\epsilon_j},
\end{equation}
\begin{equation*}
X_{\widetilde\epsilon_i+\widetilde\epsilon_1}=\frac12(X_{\epsilon_i+\epsilon_1}
+[X_{\mu},\theta X_{\epsilon_i+\epsilon_1}]+
\theta X_{\epsilon_i+\epsilon_1})
\end{equation*}
and
\begin{equation*}
X_{\widetilde\epsilon_i-\widetilde\epsilon_1}=\frac12(X_{\epsilon_i+\epsilon_1}
-[X_{\mu},\theta X_{\epsilon_i+\epsilon_1}]+
\theta X_{\epsilon_i+\epsilon_1}).
\end{equation*}
Hence,
\begin{equation}\label{ip1-im1}
X_{\widetilde\epsilon_i+\widetilde\epsilon_1}-X_{\widetilde\epsilon_i-\widetilde\epsilon_1}
=[X_{\mu},\theta X_{\epsilon_i+\epsilon_1}]=X_{\epsilon_i}\in\liem^+,
\end{equation}
Then from \eqref{ipmj} and \eqref{ip1-im1} it follows that,
\begin{equation}\label{m^+}
\liem^+ = \langle\{X_{\widetilde\epsilon_i\pm \widetilde\epsilon_j}:
2\le i<j\le4\}\cup\{X_{\widetilde\epsilon_i+
\widetilde\epsilon_1}-X_{\widetilde\epsilon_i-\widetilde\epsilon_1}:2\le
i\le 4\}\rangle,
\end{equation}
where $\langle S\rangle$ denotes the linear space spanned by the set
$S$.

\smallskip

Next we define, as in the case of Sp($n$,1) (see Section 3 of \cite{BCT}), a Lie
subalgebra $\widetilde\lieg$ of $\lieg$ that it is both $\sigma$ and $\theta$
stable and its real form $\widetilde{\lieg_o}=\lieg_o\cap\widetilde\lieg$
is isomorphic to $\liesp(2,1)$. Recall that $\alpha_1=
\tfrac12(\epsilon_1-\epsilon_2-\epsilon_3-\epsilon_4)$ is the only simple
root in $P_+$. Let $\widetilde \lieg$ be the complex Lie subalgebra
of $\lieg$ generated by the following nonzero root vectors,
$$\{X_{\pm\epsilon_2},\,
X_{\pm \alpha_1},\,
X_{\pm(\epsilon_3+\epsilon_4)}\}.$$
Then $\widetilde\lieg$ is a  simple Lie algebra
stable under $\sigma$ and $\theta$. Therefore $\widetilde\lieg$ is the
complexification of the real
subalgebra $\widetilde{\lieg_o}=\lieg_o\cap\widetilde\lieg$ and $\widetilde\lieg=
\widetilde\liek\oplus\widetilde\liep$
is a Cartan decomposition of $\widetilde\lieg$, where
$\widetilde\liek=\liek\cap\widetilde\lieg$ and
$\widetilde\liep=\liep\cap\widetilde\lieg$. Moreover
$\widetilde\lieh=(\liet\cap\widetilde\lieg)
\oplus\liea$ is a Cartan
subalgebra of $\widetilde\lieg$ and $\widetilde\liem=\liem\cap\widetilde\liek$
is the centralizer of $\liea$ in $\widetilde\liek$. That $\widetilde{\lieg_o}\simeq\liesp(2,1)$
follows from the Dynkin-Satake diagram of $\widetilde{\lieg_o}$,

\

\setlength{\unitlength}{.75mm}
\begin{picture}(0,15)
\put(55,14){$\bullet$}
\put(58,15.2){\line(1,0){18}}
\put(55,10){\scriptsize$\epsilon_2$}
\put(77,14){$\circ$}
\put(83,15.6){\line(1,0){15}}
\put(83,14.8){\line(1,0){15}}
\put(76,10){\scriptsize$\alpha_1$}
\put(80,14){$<$}
\put(99,14){$\bullet$}
\put(95,10){\scriptsize$\epsilon_3+\epsilon_4$}
\end{picture}
Since the root vectors $X_\mu$ and $\theta X_\mu$ are in $\widetilde\lieg$,
it follows that $\widetilde\lieg$ is stable under the Cayley transform
$\chi$ of the pair $(\lieg,\lieh)$. Hence the restriction of $\chi$ to
$\widetilde\lieg$ is the Cayley transform
associated to $(\widetilde\lieg,\widetilde\lieh)$. Then
$\lieh_{\widetilde\liek}=\chi(\widetilde\lieh)=\lieh_\liek\cap\widetilde\liek$
is a Cartan subalgebra of $\widetilde\liek$  and $\widetilde\lieg$.
The positive system $\Delta^+(\liek,\lieh_\liek)$ determines a positive system
$\Delta^+(\widetilde\liek,\lieh_{\widetilde\liek})=
\{\widetilde\alpha_{|_{\lieh_{\widetilde\liek}}}\in
\Delta(\widetilde\liek,\lieh_{\widetilde\liek}):\widetilde\alpha\in\Delta^+
(\liek,\lieh_\liek)\}$ in $\Delta(\widetilde\liek,\lieh_{\widetilde\liek})$.
Moreover,
\begin{equation*}
\Pi(\widetilde\liek,\lieh_{\widetilde\liek})=\{\delta=\widetilde{\epsilon_2}-
\widetilde{\epsilon_1}, \, \gamma_1=\tfrac12(\widetilde{\epsilon_1}+\widetilde{\epsilon_2}-
\widetilde{\epsilon_3}-\widetilde{\epsilon_4}), \, \gamma_2=\widetilde{\epsilon_3}+
\widetilde{\epsilon_4}\}
\end{equation*}
is a simple system in $\Delta^+(\widetilde\liek,\lieh_{\widetilde\liek})$ and the
corresponding Dynkin diagram is

\bigskip

\setlength{\unitlength}{.75mm}
\begin{picture}(0,15)
\put(60,14){$\circ$}
\put(60,10){\scriptsize$\delta$}
\put(76,14){$\circ$}
\put(82,15.7){\line(1,0){10}}
\put(82,14.9){\line(1,0){10}}
\put(76,10){\scriptsize$\gamma_1$}
\put(80,14){$<$}
\put(94,14){$\circ$}
\put(94,10){\scriptsize$\gamma_2$}
\end{picture}

\noindent Then $\Delta^+(\widetilde\liek,\lieh_{\widetilde\liek})=
\{\delta, \gamma_1,\gamma_2, \gamma_3, \gamma_4\}$, where
$\gamma_3=\gamma_1+\gamma_2=\tfrac12(\widetilde{\epsilon_1}+\widetilde{\epsilon_2}+
\widetilde{\epsilon_3}+\widetilde{\epsilon_4})$ and
$\gamma_4=2\gamma_1+\gamma_2=\widetilde{\epsilon_1}+\widetilde{\epsilon_2}$.
Hence $\widetilde\liek\simeq\liesp(1,\CC)\times\liesp(2,\CC)$.

\smallskip

A simple calculation shows that $\chi(\theta
X_{-\alpha_1})=\frac{\sqrt2}2 E$, thus $E$ is a root vector in
$\widetilde\liek^+$ corresponding to $\gamma_3$. Then set
$X_{\gamma_3}=E$. Now define
$\varphi_1=\widetilde{\epsilon_3}+\widetilde{\epsilon_1}$,
$\delta_1=\widetilde{\epsilon_3}-\widetilde{\epsilon_1}$,
$\varphi_2=\widetilde{\epsilon_4}+\widetilde{\epsilon_1}$ and
$\delta_2=\widetilde{\epsilon_4}-\widetilde{\epsilon_1}$. Then in
view of \eqref{def1} we have,
\begin{equation}\label{pd1}
X_{\gamma_4}= \chi(X_{\epsilon_2+\epsilon_1}), \quad \,\,
X_{\delta}=\chi(\theta X_{\epsilon_2+ \epsilon_1}), \quad \,\,
X_{\varphi_1}= \chi(X_{\epsilon_3+\epsilon_1}),
\end{equation}
and
\begin{equation}\label{pd2}
X_{\delta_1}= \chi(\theta X_{\epsilon_3+\epsilon_1}), \quad
X_{\varphi_2}=\chi(X_{\epsilon_4+ \epsilon_1}), \quad
X_{\delta_2}=\chi(\theta X_{\epsilon_4+\epsilon_1}).
\end{equation}
It follows from \eqref{ip1-im1} that $X_{\gamma_4}-X_{\delta}$
and $X_{\varphi_i}-X_{\delta_i}$ are in $\liem^+$ for $i=1,2$.

\smallskip

Normalize $X_{-\gamma_4}$, $X_{-\delta}$, $X_{-\varphi_i}$ and
$X_{-\delta_i}$  so that $\langle
X_{\gamma_4},\,X_{-\gamma_4}\rangle=\langle X_{\delta},\,X_{-\delta}
\rangle=\langle X_{\varphi_i},\,X_{-\varphi_i} \rangle=\langle
X_{\delta_i},\,X_{-\delta_i} \rangle=1$, for $i=1,2$. Then it follows that,
\begin{equation}\label{ortog}
\langle X_{\gamma_4}-X_{\delta},\,X_{-\gamma_4}+X_{-\delta}\rangle=
\langle X_{\varphi_i}-X_{\delta_i},\,X_{-\varphi_i}+X_{-\delta_i}\rangle =0.
\end{equation}
Hence, $X_{-\gamma_4}+X_{-\delta}$ and $X_{-\varphi_i}+X_{-\delta_i}$ $(i=1,2)$
are in $(\liem^+)^\bot$, the orthogonal complement of $\liem^+$ in $\liek$
with respect to the Killing form of $\liek$.

To simplify the notation set, $X_{\pm1}=X_{\pm\gamma_1}$, $X_{\pm2}=X_{\pm\gamma_2}$,
$X_{\pm3}=X_{\pm\gamma_3}$, and $X_{\pm4}=X_{\pm\gamma_4}$. Let
$H_1\in[\liek_{\gamma_1},\liek_{-\gamma_1}]$  be such that $\gamma_1(H_1)=2$, and
normalize $X_1$ and $X_{-1}$ so that $\{H_1,X_1,X_{-1}\}$ is an
$\lies$-triple. Next normalize $X_2$ and $X_4$ (and accordingly $X_\delta$)
so that
\begin{equation*}\label{norm1}
[X_1,X_2]=E \quad {\text {and}} \quad [X_1,E]=X_4.
\end{equation*}
From this, and the fact that $\gamma_2(H_1)=-2$, it follows that
\begin{equation*}\label{norm2}
[X_{-1},E]=2X_2 \quad {\text {and}} \quad [X_{-1},X_4]=2E.
\end{equation*}

Now choose $H_2\in[\liek_{\gamma_2},\liek_{-\gamma_2}]$ such that $\gamma_2(H_2)=2$
and normalize $X_{-2}$ so that $\{H_2,X_2,X_{-2}\}$ is an $\lies$-triple.
Since $[\liek_{\gamma_2},\liek_{-\gamma_2}]\subset\liet$ and $\gamma_1(H_2)=-1$,
if we define
\begin{equation}\label{H}
H=\tfrac12H_2,
\end{equation}
we obtain a vector $H\in\liet$ such that $\dot H(E)=\tfrac12E$. This vector $H$ is the
one used in \eqref{c2}. Also, since $\delta(H_2)=0$, we have
$[X_{\delta},H]=0$.

\smallskip

As in the previous sections set $Z=Z_{\alpha_1}$, $Y=Y_{\alpha_1}$
and $\widetilde Y=Y+H$. From Lemma \ref{defc} it follows that $\dot
E(Y)=\frac32 E$, hence $\dot E(\widetilde Y)=E$. Now, since
$(\epsilon_1+\epsilon_2)(H_{\alpha_1})=0$, we have
$(\epsilon_1+\epsilon_2)(Y)=-(\epsilon_1+\epsilon_2)(Z)=-1$ because
$(\epsilon_1+\epsilon_2)|_\liea=2\alpha_1|_\liea$ and
$\alpha_1(Z)=\frac12$ (see Lemma \ref{defc}). Then $\dot
X_{\delta}(Y)=X_{\delta}$, and therefore $\dot X_{\delta}(\widetilde
Y)=X_{\delta}$.

\section{The $M$-spherical $K$-modules}\label{KMsph}

In this section we describe the main properties of the $K$-modules in the
classes $\Gamma$ and $\Gamma_1$ (see \eqref{Gamma1}). In the following
proposition we collect several results that will be very useful later on, and
in Proposition \ref{d(uv)} we will prove some important properties of the
Kostant degree $d(u)$ for $u\in U(\liek)^M$ that make use of these results.

\begin{prop}\label{hw3} Let $G_o$ be locally isomorphic to F$_4$ and let
$\lieb_\liek=\lieh_\liek\oplus\liek^+$ be the  Borel subalgebra of
$\liek$ defined before. Then $\liem^+\subset\liek^+$ and $E$ is a
root vector in $\liek^+$. Moreover:

\noindent (i) For any $\gamma\in\hat K$ let $\xi_\gamma$ denote its
highest  weight. Then, $\gamma\in\Gamma$ if and only if
$\xi_\gamma=\frac k2(\gamma_4+\delta)+\ell\gamma_3$ with $k, \ell\in
\NN_o$. In this context we write $\gamma=\gamma_{k,\ell}$,
$\xi_\gamma=\xi_{k,\ell}$ and $V_{k,\ell}$ for the corresponding
representation space. Also we shall refer to any $v\in V_{k,\ell}^M$
as an $M$-invariant element of type $(k,\ell)$.

\noindent(ii) For any $\gamma_{k,\ell}\in\Gamma$ we have
$d(\gamma_{k,\ell})=k+2\ell$.

\noindent(iii) If $\gamma\in\Gamma$ we have $\gamma\in\Gamma_1$ if
and only if $\xi_\gamma=\xi_{k,\ell}$ with $k$ even.

\noindent(iv)  For any $\gamma_{k,\ell}\in\Gamma$ we have
$X_\delta^kE^\ell(V^M_{k,\ell})=V^{\liek^+}_{k,\ell}$ and
$X_\delta^pE^q(V^M_{k,\ell})=\{0\}$ if  and only if $p>k$ or
$p+q>k+\ell$.
\end{prop}

For a proof of this proposition we refer the reader to \cite{CT}. The
construction of the Borel subalgebra $\lieb_\liek$ is
contained in Section 3 of \cite{CT} and the statements in (i), (ii)
and (iv) follow from Proposition 4.4, Theorem 4.5 and Theorem 5.3 of
\cite{CT}, respectively. On the other hand (iii) is a well known
general fact. We point out that some of these results where first
established in \cite{JW}, others where proved in \cite{BT} and they
were generalized in \cite{CT} to any real rank one semisimple Lie
group.

The following proposition is the analogue of part (ii) of Proposition 3.11
of \cite{BCT}. We omit its proof since, up to minor changes, is the same as
that of Proposition 3.11.

\begin{prop}\label{Techo} Let $G_o$ be locally isomorphic to
F$_4$. Let $\gamma_{k,\ell}\in\Gamma$ and let $V_{k,\ell}$ be a $K$-module
in the class $\gamma_{k,\ell}$. Then if $0\ne v\in V_{k,\ell}^M$ the set
$$\{X_{\delta}^{k-j}E^{\ell+j}(v):\,0\le j\le k\}$$
is a basis of the irreducible $\{H_1,X_1,X_{-1}\}$-module of dimension
$k+1$ generated by any non trivial highest weight vector of
$V_{k,\ell}$. Moreover, $X_{\delta}^{k-j}E^{\ell+j}(v)$ is a weight vector of weight
$\xi_{k,\ell}-j\gamma_1$ and the following identities hold,
\begin{equation}\label{Dk11}
X_1X_{\delta}^{k-j}E^{\ell+j}(v)=\frac{(j+\ell)}2X_{\delta}^{k-j+1}E^{\ell+j-1}(v)\qquad
0\le j\le k,
\end{equation}
\begin{equation}\label{Dk22}
X_{-1}X_{\delta}^{k-j}E^{\ell+j}(v)=\frac{2(j+1)(k-j)}{\ell+j+1}X_{\delta}^{k-j-1}E^{\ell+j+1}(v)
\qquad 0\le j\le k,
\end{equation}
\begin{equation}\label{Dk33}
X^j_{-1}(u_{k,\ell})=2^jj!\binom kj\binom{\ell+j}\ell^{-1}
X_{\delta}^{k-j}E^{\ell+j}(v) \qquad 0\le j\le k,
\end{equation}
where $u_{k,\ell}$ is the highest weight vector $X_{\delta}^kE^\ell(v)$.
\end{prop}

\smallskip

In the following proposition we prove some important properties of the Kostant
degree $d(u)$ for $u\in U(\liek)^M$. Even though we give the proof for F$_4$,
since our argument relies heavily on Proposition \ref{hw3}, the same proof hold
for the other real rank one groups, SO($n$,1), SU($n$,1) and Sp$(n,1)$, with the
appropriate changes. These result will be used in Section \ref{mainproof}.

\begin{prop}\label{d(uv)} Let $G_o$ be locally isomorphic to F$_4$.
If $u,v\in U(\liek)^M$ are nonzero vectors, then
\begin{enumerate}[\rm (a)]
\item $d(u+v)\le\max\{d(u),d(v)\}$,

\item $d(uv)=d(u)+d(v)$,

\item $d(u)=0$ if and only if $u\in U(\liek)^K$.
\end{enumerate}
\end{prop}

\begin{proof} The assertions (a) and (c) follow directly from the definition of the Kostant degree. We start the proof of (b) by showing  that $d(uv)\le d(u)+d(v)$ for any
$0\ne u,v\in U(\liek)^M$. Let us begin by considering $u\in V_{r,s}\subset
U(\liek)^M$ and $v\in V_{r',s'}\subset U(\liek)^M$ where $V_{r,s}$ and
$V_{r',s'}$ are, respectively, irreducible finite dimensional $K$-modules
in the classes $\gamma_{r,s}$ and $\gamma_{r',s'}$ of $\Gamma_1$. Then
$u\otimes v\in(V_{r,s}\otimes V_{r',s'})^M$ and we decompose it as follows,
\begin{equation}\label{uxv}
u\otimes v=\sum_{i,j}{\textbf w}_{i,j},
\end{equation}
where ${\textbf w}_{i,j}\ne0$ is the $\gamma_{i,j}$-isotypic component of $u\otimes v$. We recall that if
$\gamma_{i,j}\in \Gamma$ then its  highest weight is  $\xi_{i,j}=\frac {i}{2}\,(\gamma_4+\delta)+j\,\gamma_3$ and $d(\gamma_{i,j})=i+2j$,
see Proposition \ref{hw3}. We will show that $d({\textbf w}_{i,j})\le d(u)+d(v)$
for any ${\textbf w}_{i,j}$ that occurs in \eqref{uxv}.

\smallskip

In view of \eqref{simroot} a simple system of roots in $\Delta^+(\liek,\lieh_\liek)$
is given by,
\begin{equation}\label{simroot(k)}
\Pi(\liek,\lieh_\liek)=\{\widetilde\epsilon_4+\widetilde\epsilon_1, \,
\widetilde\epsilon_3-\widetilde\epsilon_4, \,
\widetilde\epsilon_4-\widetilde\epsilon_1, \,
\gamma_1=\tfrac12(\widetilde\epsilon_1+\widetilde\epsilon_2-
\widetilde\epsilon_3-\widetilde\epsilon_4)\}.
\end{equation}
Then it follows that
\begin{equation*}
\gamma_4+\delta=(\widetilde\epsilon_4+\widetilde\epsilon_1)+
2(\widetilde\epsilon_3-\widetilde\epsilon_4)+
3(\widetilde\epsilon_4-\widetilde\epsilon_1)+ 4\,\gamma_1
\end{equation*}
and
\begin{equation*}
\gamma_3=(\widetilde\epsilon_4+\widetilde\epsilon_1)+
(\widetilde\epsilon_3-\widetilde\epsilon_4)+
(\widetilde\epsilon_4-\widetilde\epsilon_1)+ \gamma_1.
\end{equation*}

\smallskip

\noindent If $V_{i,j}\subset U(\liek)^M$ occurs in the decomposition of
$V_{r,s}\otimes V_{r',s'}$ it is known (see \cite {Hu}) that its highest
weight $\xi_{i,j}=\frac {i}{2}\,(\gamma_4+\delta)+j\,\gamma_3$
is given by,
\begin{equation}\label{xi_{i,j}}
\xi_{i,j}=\xi_{r+r',s+s'}-\left[c_1(\widetilde\epsilon_4+\widetilde\epsilon_1)+
c_2(\widetilde\epsilon_3-\widetilde\epsilon_4)+
c_3(\widetilde\epsilon_4-\widetilde\epsilon_1)+
c_4\,\gamma_1\right],
\end{equation}

\noindent where $c_i \in\NN_o$ for $1\le i\le 4$. Hence comparing the
coefficients of the simple root $\widetilde\epsilon_4+\widetilde\epsilon_1$ in the
left hand side and the right hand side of \eqref{xi_{i,j}} it follows that
$$\frac{i}2 + j=\frac{r+r'}2 +s+s' - c_1.$$
Then, since $c_1\ge 0$, we have
$$d(\textbf{w}_{i,j})= r+r'+2(s+s')-2c_1=
d(u)+d(v)-2c_1 \le d(u)+d(v).$$
Therefore, using the definition \eqref{Kostdeg} and
\eqref{uxv}  it follows that,
$$d(u\otimes v) = \max \{d(\textbf{w}_{i,j})\}\le d(u)+d(v).$$
Now, using that the map $u\otimes v\in U(\liek)^M\otimes U(\liek)^M
\rightarrow uv\in U(\liek)^M$ is a $K$-homomorphism it follows that
$d(uv) \le d(u)+d(v)$.

\smallskip

Now let $u\in V_{r,s}\oplus \dots \oplus V_{r,s}$ ($m$ summands) and
$v\in V_{r',s'}\oplus\dots\oplus V_{r',s'}$ ($n$ summands), where
$V_{r,s}$ and $V_{r',s'}$ are irreducible finite dimensional $K$-submodules of
$U(\liek)^M$ as above. Write $u=u_1+\dots+u_m$ with $u_k\in V_{r,s}$
($1\le k \le m$) and $v=v_1+\dots+v_n$ with $v_{\ell}\in V_{r',s'}$
($1\le \ell \le n$). Then using above calculation we obtain,
\begin{equation}\label{d(uv)'}
\begin{split}
d(uv)&=d\Big(\!\sum_{k,\ell}u_kv_{\ell}\Big)
\le\max \{d(u_kv_{\ell}): 1\le k\le m,\,\, 1\le\ell\le n\}\\
             &\le\max \{d(u_k)+d(v_{\ell})): 1\le k\le m,\,\, 1\le\ell\le n\}
             =d(u)+d(v).\\
\end{split}
\end{equation}

\smallskip

Consider now $u,v\in U(\liek)^M$ such that $d(u)=p$ and $d(v)=q$. It follows
from \eqref{Kostdeg} that,
\begin{equation}\label{K-isotypic}
u=\sum_{d(\gamma)\le p} \textbf{u}_{\gamma} \quad\quad \text{and} \quad\quad
v=\sum_{d(\tau)\le q} \textbf{v}_{\tau},
\end{equation}
where $\textbf{u}_{\gamma}$ and $\textbf{v}_{\tau}$ denote, respectively, the
$K$-isotypic componets of $u$ and $v$ corresponding to the classes $\gamma$ and
$\tau$ of $\Gamma_1$. Then using \eqref{d(uv)'} we obtain,
\begin{equation*}\label{d(uv)1}
\begin{split}
d(uv)&=d\Big(\!\sum_{\gamma,\tau}\textbf{u}_{\gamma}\textbf{v}_{\tau}\Big)
\le\max \{d(\textbf{u}_{\gamma}\textbf{v}_{\tau}): \textbf{u}_{\gamma}\neq 0,
\, \textbf{v}_{\tau}\neq 0 \}\\
&\le\max \{d(\textbf{u}_{\gamma})+d(\textbf{v}_{\tau}): \textbf{u}_{\gamma}\neq 0,
\, \textbf{v}_{\tau}\neq 0 \}\\
&=\max \{d(\gamma)+d(\tau): \textbf{u}_{\gamma}\neq 0,\, \textbf{v}_{\tau}\neq 0 \}\\
 &\le p+q = d(u)+d(v).\\
\end{split}
\end{equation*}

\smallskip

Our next goal is to show that $d(uv)=d(u)+d(v)$ for any $u,v\in U(\liek)^M$.
Assume that $d(u)=p$ and $d(v)=q$. Then, using \eqref{K-isotypic} and the fact that
$d(uv)\le d(u)+d(v)$ for any $u,v\in U(\liek)^M$ it follows that,
\begin{equation*}
uv=\sum_{d(\gamma)=p, \,\, d(\tau)=q} \textbf{u}_{\gamma}\textbf{v}_{\tau} + w,
\end{equation*}
where $w\in U(\liek)^M$ is such that $d(w)< p+q$. Then, in view of \eqref{Kostdeg}
we may assume that
\begin{equation}\label{u,v}
u=\sum_{i+2j=p} \textbf{u}_{i,j} \quad\quad \text{and} \quad\quad
v=\sum_{r+2s=q} \textbf{v}_{r,s},
\end{equation}
where $\textbf{u}_{i,j}$ and $\textbf{v}_{r,s}$ denote, respectively, the
$K$-isotypic components of $u$ and $v$ corresponding to the classes $\gamma_{i,j}$
and $\gamma_{r,s}$ of $\Gamma_1$. Let $k=\max \{i\in \NN_o: \textbf{u}_{i,j}\neq 0
\,\,\text{for some}\,\,j \}$ and $\ell=\max \{r\in \NN_o: \textbf{v}_{r,s}\neq 0
\,\,\text{for some}\,\,s \}$. Then using \eqref{u,v}, Leibnitz rule and part (iv) of
Proposition \ref{hw3} it follows that,
\begin{equation}\label{XE(uv)}
\begin{split}
\dot E&^{(p+q-k-\ell)/2}\dot X_{\delta}^{k+\ell}(uv)\\
&=\binom{k+\ell}{\ell}\binom{\frac{p+q-k-\ell}{2}}{\frac{q-\ell}{2}}
\,\dot E^{(p-k)/2}\dot X_{\delta}^{k}(\textbf{u}_{k,\frac{p-k}2})\,
\dot E^{(q-\ell)/2}\dot X_{\delta}^{\ell}(\textbf{v}_{\ell,\frac{q-\ell}2})\ne0.
\end{split}
\end{equation}
We point out that the right hand side of \eqref{XE(uv)} is different from zero
because, in view of (iv) of Proposition \ref{hw3}, it is a product of two dominant
vectors. Also using Leibnitz rule, Proposition \ref{hw3} (iv) and \eqref{XE(uv)}
it follows that,
\begin{equation}\label{X(uv)1}
\dot X_{\delta}^{k+\ell}(uv)=\binom{k+\ell}{\ell}\dot X_{\delta}^{k}(\textbf{u}_{k,\frac{p-k}2})\,
\dot X_{\delta}^{\ell}(\textbf{v}_{\ell,\frac{q-\ell}2})\ne0,
\end{equation}
and
\begin{equation}\label{X(uv)2}
\dot X_{\delta}^{k+\ell+1}(uv)=0.
\end{equation}

\smallskip

To finish the proof write
\begin{equation*}\label{uv}
uv=\sum_{i,j}{\textbf b}_{i,j},
\end{equation*}
where $\textbf{b}_{i,j}$  denote the
$K$-isotypic components of $uv$ corresponding, respectively, to the classes
$\gamma_{i,j}\in \Gamma_1$. Then from \eqref{XE(uv)}, \eqref{X(uv)1} and \eqref{X(uv)2}
we obtain,
$$\dot X_{\delta}^{k+\ell}(uv)=\sum_j\dot X_{\delta}^{k+\ell}({\textbf b}_{k+\ell,j})$$
and
$$0\ne\sum_j\dot E^{(p+q-k-\ell)/2}\dot X_{\delta}^{k+\ell}({\textbf b}_{k+\ell,j}).$$
Therefore, from Proposition \ref{hw3} (iv) it follows that there exists
${\textbf b}_{k+\ell,j}\neq 0$ such that $(p+q-k-\ell)/2+k+\ell\le k+\ell+j$. Thus
$$d(uv)\le d(u)+d(v)=p+q\le k+\ell+2j=d({\textbf b}_{k+\ell,j})\le d(uv).$$
This completes the proof of the proposition. \qed
\end{proof}

\section{Transversality results}\label{trans.}

In this section we prove several results that will allow us to deal with
the congruence modulo $U(\liek)\liem^+$ that occur in the equations that
define the algebra $B$ (see \eqref{eqB}). In particular, we reduce
the congruence modulo $U(\liek)\liem^+$ to a congruence modulo
$U(\liek)\liey$, where $\liey\subset \liem^+$ is the abelian subalgebra
defined as follows
\begin{equation}\label{n}
\liey = \langle\{X_{\widetilde{\epsilon_3}+\widetilde{\epsilon_4}},
X_{\widetilde{\epsilon_2}+\widetilde{\epsilon_3}},
X_{\widetilde{\epsilon_2}+\widetilde{\epsilon_4}}\}\rangle.
\end{equation}
Before stating the main results we introduce the following notation,
\begin{equation}\label{ST}
S_{23}= X_{\widetilde{\epsilon_2}+\widetilde{\epsilon_3}}, \quad
S_{24}= X_{\widetilde{\epsilon_2}+\widetilde{\epsilon_4}},\quad \text{and}
\quad T_{ij}= X_{\widetilde{\epsilon_i}-\widetilde{\epsilon_j}} \quad
(2\le i\neq j \le4).
\end{equation}

Let $\lieq^+$ be the linear span of
$\{X_{\alpha}:\alpha\in\Delta^+(\liek,\lieh_\liek)
\text{\,and\,\,}\alpha\ne\gamma_1 \}$. Since $\gamma_1$ is a simple
root in $\Delta^+(\liek,\lieh_\liek)$ (see \eqref{simroot(k)}) it follows that $\lieq^+$
is a subalgebra of $\liek^+$. We are interested in considering weight
vectors $u\in U(\liek)\liem^+$ of weight $\lambda=a(\gamma_4+\delta)+b\gamma_3$
($a,b\in\ZZ$), and such that $\dot X(u) \equiv 0 \mod(U(\liek)\liey)$
for every $X\in \lieq^+$.

\smallskip

Consider the subalgebra $\lieq \subset \liek$ defined as follows
\begin{equation}\label{q}
\lieq=\lieq^+\oplus\lieh_\lier\oplus\lieq^{-},
\end{equation}
where
\begin{equation}\label{hr}
\lieh_{\lier}=\ker(\gamma_4+\delta)\cap \ker(\gamma_3)=\langle\{H_{\widetilde{\epsilon_3}-\widetilde{\epsilon_4}}, H_{\widetilde{\epsilon_4}-\widetilde{\epsilon_1}}\}\rangle
\end{equation}
and
\begin{equation}\label{q^-}
\lieq^-=\langle\{X_{-(\,\widetilde{\epsilon_3}-\widetilde{\epsilon_4})}\}\rangle.
\end{equation}
Then a simple calculation shows that,
\begin{equation*}\label{[q,n]}
[\lieq,\liey]\subset \liey.
\end{equation*}
Moreover, $\lieq=\lier\oplus\lieu$ where $\lier= \langle \lieh_{\lier} \cup
\{X_{\pm(\,\widetilde{\epsilon_3}-\widetilde{\epsilon_4})}\}\rangle\simeq\liegl(2,\CC)$, $\lieh_{\lier}$ is a Cartan subalgebra of $\lier$ and $\lieu$ is the following nilpotent
subalgebra,
$$\lieu = \langle
\{X_{\widetilde{\epsilon_2}\pm \widetilde{\epsilon_j}}: 3\le j\le
4\}\cup \{X_{\widetilde{\epsilon_i}\pm\widetilde{\epsilon_1}}: 2\le
i\le4 \} \cup \{X_{\gamma_2}, X_{\gamma_3},
X_{\psi_1},X_{\psi_2}\}\rangle,$$
where
\begin{equation}\label{psi}
\psi_1= \tfrac 12(-\widetilde{\epsilon_1}+
\widetilde{\epsilon_2}- \widetilde{\epsilon_3}+
\widetilde{\epsilon_4}), \quad \quad \psi_2=\tfrac
12(-\widetilde{\epsilon_1}+ \widetilde{\epsilon_2}+
\widetilde{\epsilon_3}- \widetilde{\epsilon_4}).
\end{equation}

\smallskip

The proof of the next two lemmas follow from a direct application of
Poincar\'e-Birkhoff-Witt theorem. Let $\lieg$ be an arbitrary finite
dimensional complex Lie algebra and let $\liel$ be a subalgebra of
$\lieg$. If $\{X_1,\dots,X_p\}$ is an ordered basis of $\liel$
complete it to an ordered basis $\{Y_1,\dots,Y_q,X_1,\dots,X_p\}$ of
$\lieg$. Now, if $I=(i_1,\dots,i_q)\in \NN_o^q$ and
$J=(j_1,\dots,j_p)\in \NN_o^p$ define as usual $Y^IX^J=Y_1^{i_1}\cdots
Y_q^{i_q}X_1^{j_1}\cdots X_p^{j_p}$ in $U(\lieg)$. Then we have,
\begin{lem}\label{U(g)l1} Any $u\in U(\lieg)\liel$ can be written in a
unique way as $u=a_1X_1+\cdots+ a_pX_p$ where
$$a_k=\sum a_{I,j_1,\dots,j_k}\,Y^IX_1^{j_1}\cdots X_k^{j_k}
\quad\textstyle{for}\quad k=1,\dots,p,$$
and the coefficients $a_{I,j_1,\dots,j_k}$ are complex numbers.
\end{lem}
\begin{lem}\label{U(g)l2} Let $\lieg$ and $\liel$ be as above. Let $u\in
U(\lieg)$ and $X\in\lieg-\liel$ be such that $\dot
X(\liel)\subset\liel$. If $uX^n\equiv0\mod(U(\lieg)\liel)$ for some
$n\in\NN$, then $u\equiv0\mod(U(\lieg)\liel).$
\end{lem}

Let $\liey^\bot$ be the orthogonal complement of $\liey$ in $\liek$ with
respect to the Killing form of $\liek$. For any $Z\in (\liem^+)^\bot$
consider the linear map $T_Z:\lieq \times (\liem^+)^\bot \rightarrow \liey^\bot$
given by
\begin{equation}\label{Tz}
T_Z(X,Y)=[X,Z]+Y, \quad \quad X \in \lieq \quad \text{and}\quad
Y\in (\liem^+)^\bot .
\end{equation}
Since $[\lieq,\liey]\subset \liey$ and $(\liem^+)^\bot\subset\liey^\bot$ it
follows that Im$(T_Z)\subset\liey^\bot$, where Im$(T_Z)$ denotes the image of
the map $T_Z$. The following proposition will be used in Theorem \ref{m+/n}
to prove one of the main results of this section.

\begin{prop}\label{surj} There exists $Z_o\in(\liem^+)^\bot$ such that
\rm{Im}$(T_{Z_o})=\liey^\bot$.
\end{prop}

\begin{proof} Using \eqref{m^+} and the notation introduced in \eqref{pd1},
\eqref{pd2} and \eqref{ST} it is easy to check that,
\begin{equation}\label{n^bot}
\liey^\bot=(\liem^+)^\bot\oplus\langle \{X_{-\delta},\,X_{-\delta_1},\,
X_{-\delta_2},\,T_{32},\,T_{42},\,T_{43}\}\rangle.
\end{equation}

It is clear, from the definition of $T_Z$, that $(\liem^+)^\bot\subset \text{Im}(T_Z)$
for every $Z\in (\liem^+)^\bot$. Now consider the vector,
\begin{equation}\label{Z_o}
Z_o = X_{-\gamma_4} + X_{-\delta} + X_{-\varphi_2} + X_{-\delta_2} +
X_{-\gamma_3} + H_{\widetilde{\epsilon_4}-\widetilde{\epsilon_3}},
\end{equation}
where $H_{\widetilde{\epsilon_4}-\widetilde{\epsilon_3}}\in \lieh_{\liek}$ is
such that $(\widetilde{\epsilon_4}-\widetilde{\epsilon_3})
(H_{\widetilde{\epsilon_4}-\widetilde{\epsilon_3}})=2$. Using \eqref{m^+} and
\eqref{ortog} it follows that $Z_o\in(\liem^+)^\bot$.
In view of \eqref{n^bot}, to prove that Im$(T_{Z_o})=\liey^\bot$ we need to show
that $\langle \{X_{-\delta},\,X_{-\delta_1},\, X_{-\delta_2},\,T_{32},\,T_{42},
\,T_{43}\}\rangle$ is contained in Im$(T_{Z_o})$. In fact, using that
$X_{\varphi_1}, X_{\varphi_2}, X_{\psi_1}, X_{\psi_2}, H_{\widetilde{\epsilon_4}-\widetilde{\epsilon_1}}$ and $T_{43}$ are in $\lieq$
(see \eqref{pd1}, \eqref{pd2}, \eqref{ST} and \eqref{psi} for the notation) a simple
calculation shows that,
\begin{equation*}
T_{Z_o}(X_{\varphi_2},0)\equiv c_1\, T_{42}, \quad \quad \quad T_{Z_o}(X_{\varphi_1},0)
\equiv c_2\, T_{32},
\end{equation*}

\begin{equation*}
T_{Z_o}(X_{\psi_2},0)\equiv c_3\, X_{-\delta_2},\quad \quad \quad T_{Z_o}(X_{\psi_1},0)
\equiv c_4\, X_{-\delta_1},
\end{equation*}

\begin{equation*}
T_{Z_o}(H_{\widetilde{\epsilon_4}-\widetilde{\epsilon_1}},0)\equiv c_5\,
X_{-\delta},\quad \quad \quad T_{Z_o}(T_{43},0)\equiv c_6\, T_{43},
\end{equation*}
where, in all cases, the congruence is module the subspace $(\liem^+)^\bot$ and
$c_i\neq 0$ for $1\leq i\leq 6$. This completes the proof of the proposition.
\qed
\end{proof}

\begin{thm}\label{m+/n} Let $u\in U(\liek)\liem^+$ be a vector of weight
\mbox{$\lambda=a(\gamma_4+\delta)+b\gamma_3$}, with $a,b\in\ZZ$, and
such that $\dot X(u) \equiv 0 \mod(U(\liek)\liey)$ for every $X \in \lieq^+$.
Then $u\equiv0\mod\!(U(\liek)\liey)$.
\end{thm}

\begin{proof}Let $U(\liek)=\bigcup_{j\ge 0}U_j(\liek)$ be the canonical
ascending filtration of $U(\liek)$. If $v \in U(\liek)$ and $v\ne 0$ define,
\begin{equation}\label{deg}
\text{deg}(v)= \min \{j: v\in U_j(\liek)\,\, \text{and}\,\, v\notin
U_{j-1}(\liek)\},
\end{equation}
where it is understood that  $U_{-1}(\liek)=\{0\}$.
Let $\calS$ be the set of all $v \in U(\liek)\liem^+$ of weight $\lambda =
a(\gamma_4+\delta)+b\gamma_3$ ($a,b\in\ZZ$), so that $\dot X(v)\in U(\liek)\liey$
for every $X \in \lieq^+$ and $v\notin U(\liek)\liey$. The theorem will be
proved if we show that $\calS = \emptyset$. Assume on the contrary that
$\calS \neq \emptyset$ and choose $u\in \calS$ such that
$\text{deg}(u)= \min \{\text{deg}(v): v\in \calS\}$. Set $r=\text{deg}(u)$
and let $p_r:U_r(\liek) \rightarrow U_r(\liek)/U_{r-1}(\liek)$ denote the
quotient map. The map $p_r$ intertwines the representations of $K$ on $U_r(\liek)$
and on $U_r(\liek)/U_{r-1}(\liek)$, and since $u\notin U_{r-1}(\liek)$ we have
$p_r(u)\ne 0$.

\smallskip

Let $S(\liek)$ be the symmetric algebra of $\liek$ and let $S(\liek^*)$ denote
the algebra of polynomial functions on $\liek$. Let $S_r(\liek)$ and
$S_r(\liek^*)$ denote the corresponding homogeneous subspaces of $S(\liek)$ and
$S(\liek^*)$ of degree $r$. There is an algebra isomorphism between
$S(\liek)$ and $S(\liek^*)$ defined by the Killing form of $\liek$, this
isomorphism maps $S_r(\liek)$ onto $S_r(\liek^*)$ and intertwines the canonical
representations of $K$ on $S_r(\liek)$ and on $S_r(\liek^*)$. Composing this
isomorphism with the natural $K$-isomorphism between $U_r(\liek)/U_{r-1}(\liek)$ and
$S_r(\liek)$ we obtain a $K$-isomorphism,
\begin{equation}\label{K-iso}
U_r(\liek)/U_{r-1}(\liek)\simeq S_r(\liek^*).
\end{equation}
Hence we can think of $p_r(u)$ as a homogeneous polynomial function
on $\liek$ of degree r, and regard $p_r$ as a $K$-homomorphism from $U_r(\liek)$
to $S_r(\liek^*)$.

\smallskip

Let $(\liem^+)^\bot$ be the orthogonal complement of $\liem^+$ in $\liek$
with respect to the Killing form of $\liek$. Since $u \in U(\liek)\liem^+$
and the isomorphism given in \eqref{K-iso} is defined by the Killing form of
$\liek$ it follows that,
\begin{equation}\label{p(u)(Y)}
p_r(u)(Y)=0 \quad \text{for every}\quad Y \in (\liem^+)^\bot.
\end{equation}

Now let $X\in \lieq^+$. Since $[\lieq^+,\liey]\subset \liey$ we have
$\dot X^k(U(\liek)\liey)\subset U(\liek)\liey$ for every $k\in \NN$. Then, since
by hypothesis $\dot X(u)\in U(\liek)\liey$, it follows that $\dot X^k(u)\in U(\liek)\liey$
for any $k\in \NN$. Therefore, using that $(\liem^+)^\bot\subset \liey^\bot$ and
that $p_r$ is a $K$-homomorphism it follows by induction on $k$ that
\begin{equation}\label{X^k(p(u))}
X^k(p_r(u))(Y)=p_r(\dot X^k(u))(Y)=0 \,\,\, \text{for}\,\,\, Y \in (\liem^+)^\bot
\,\,\, \text{and}\,\,\, X \in \lieq^+,
\end{equation}
where $X(p_r(u))$ denotes the action of $X$ on the polynomial function
$p_r(u)$.

\smallskip

Since $u$ is a vector of weight $\lambda=a(\gamma_4+
\delta)+b\gamma_3$, it follows from the definition of $\lieh_{\lier}$ that
$\dot H(u)=0$ for every $H\in\lieh_{\lier}$. Then,
\begin{equation}\label{H^k(p(u))}
H^k(p_r(u))(Y)=0 \,\,\, \text{for} \,\,\, Y\in \liek, \,\,\,
H\in\lieh_{\lier} \,\,\, \text{and}\,\,\, k\in \NN.
\end{equation}

\smallskip

Let $0\ne\overline{u}\in U(\liek)/U(\liek)\liey$ be the image of $u$ under the
quotient map. Normalize $X_{\widetilde{\epsilon_3}-\widetilde{\epsilon_4}}$
and $X_{-(\,\widetilde{\epsilon_3}-\widetilde{\epsilon_4})}$ so that
$\{X_{\widetilde{\epsilon_3}-\widetilde{\epsilon_4}},H_{\widetilde{\epsilon_3}
-\widetilde{\epsilon_4}}, X_{-(\,\widetilde{\epsilon_3}-\widetilde{\epsilon_4})}\}$
is an $\lies$-triple.  Since $X_{\widetilde{\epsilon_3}-\widetilde{\epsilon_4}}
\in \lieq^+$ and  $H_{\widetilde{\epsilon_3}-\widetilde{\epsilon_4}}\in \lieh_{\lier}$,
and by hypothesis $\overline{u}$ is a dominant vector of weight zero with
respect to above $\lies$-triple, we obtain that
$\dot X_{-(\,\widetilde{\epsilon_3}-\widetilde{\epsilon_4})}(\overline{u})=0$.
Hence, from \eqref{q^-} we obtain that $\dot X(u)\in U(\liek)\liey$ for $X\in
\lieq^{-}$. Then, since $[\lieq^-,\liey]\subset \liey$, it follows that
\begin{equation}\label{X_^k(p(u))}
X^k(p_r(u))(Y)=0 \,\,\, \text{for}\,\,\, Y \in (\liem^+)^\bot,
\,\,\, X \in \lieq^- \,\,\, \text{and}\,\,\, k\in \NN.
\end{equation}

\smallskip

Now recall that for $k\in K$ and $f\in S_r(\liek^*)$ the action of $k$ on $f$ is
given by $(kf)(Y)=f(Ad(k^{-1})Y)$ for every $Y \in \liek$ . Then, from \eqref{p(u)(Y)}, \eqref{X^k(p(u))}, \eqref{H^k(p(u))} and \eqref{X_^k(p(u))} it follows that
\begin{equation}\label{p(u)(Ad(X))}
p_r(u)(Ad(\exp X)Y)=0 \quad \text{for}\quad X \in \lieq^+ \cup \lieh_{\lier} \cup \lieq^-
\quad \text{and}\quad Y\in (\liem^+)^\bot .
\end{equation}
Let $Q$ be the connected Lie subgroup of $K$ with Lie algebra $\lieq$ (see \eqref{q}).
Since the set $\exp\lieq^+.\exp\lieh_{\lier}.\exp\lieq^-$ generates $Q$ we obtain that,
\begin{equation}\label{Q}
p_r(u)(Ad(g)Y)=0 \quad \text{for}\quad g \in Q \quad \text{and}\quad
Y\in (\liem^+)^\bot .
\end{equation}

\smallskip

Now consider the map $\Phi:Q \times (\liem^+)^\bot \rightarrow \liey^\bot$ defined by
$\Phi(g,Y)=Ad(g)Y$. The fact that the image of $\Phi$ is contained in $\liey^\bot$
follows from a simple calculation using that $[\lieq,\liey]\subset \liey$ and that
$\liey \subset \liem^+$. Let $e\in Q$ be the identity element and $Z\in (\liem^+)^\bot$,
then $(d\Phi)_{(e,Z)}$ is the map $T_Z:\lieq \times (\liem^+)^\bot \rightarrow \liey^\bot$
defined in \eqref{Tz}. It follows from Proposition \ref{surj} that $(d\Phi)_{(e,Z_o)}$ is  surjective. This implies that the image of $\Phi$ contains an open set of $\liey^\bot$,
then in view of \eqref{Q} we obtain that,
\begin{equation}\label{p(u)=0}
p_r(u)(Y)=0 \quad \text{for every}\quad Y\in \liey^\bot.
\end{equation}

\smallskip

Recall that $\liey = \langle\{X_2,S_{23},S_{24}\}\rangle$ (see \eqref{n}). Extend the
basis of $\liey$  to a basis $\calB=\{Z_1,\cdots,Z_q,X_2,S_{23},S_{24}\}$ of $\liek$,
where $q=\dim\liek-3$.
If $I=(i_1,\dots,i_q)\in \NN_o^q$ and $J=(j_1,j_2,j_3)\in \NN_o^3$, set
$|I|=i_1+\dots+i_q$, $|J|=j_1+j_2+j_3$ and $Z^I=Z_1^{i_1}\dots Z_q^{i_q}$
in $S(\liek)$. If we regard $p_r(u)$ as an element in $S_r(\liek)$ we can write
\begin{equation*}
p_r(u)=\sum b_{I,J}\, Z^I X_2^{j_1}S_{23}^{j_2}S_{24}^{j_3},
\end{equation*}
where $b_{I,J}\in \CC$ and the sum extends over all $I$ and $J$ such that
$|I|+|J|=r$. Now, identifying $\liek^*$ with $\liek$ via the Killing form of
$\liek$ and considering a basis $\widetilde\calB$ of $\liek$ dual to $\calB$
it follows from \eqref{p(u)=0} that $b_{I,0}=0$, for all $I$ such that $|I|=r$.
Therefore
\begin{equation}\label{p(u)}
p_r(u)=\sum_{|J|>0} b_{I,J}\, Z^I X_2^{j_1}S_{23}^{j_2}S_{24}^{j_3},
\end{equation}
where the sum extends over all $I$ and $J$ such that $|I|+|J|=r$. On the other hand,
since $p_r$ is a $K$-homomorphism from $U_r(\liek)$ to $S_r(\liek)$ it follows that
$p_r(u)$ has weight $\lambda=a(\gamma_4+\delta)+b\gamma_3$ with respect to
$\lieh_{\liek}$. Then, \eqref{p(u)} implies that
\begin{equation}\label{u'}
u=\sum_{|J|>0} b_{I,J}\, Z^I X_2^{j_1}S_{23}^{j_2}S_{24}^{j_3} + u',
\end{equation}
where the monomials $Z^I X_2^{j_1}S_{23}^{j_2}S_{24}^{j_3}$ are in $U(\liek)$, the
sum extends over all $I$ and $J$ such that $|I|+|J|=r$ and $u'$ is a vector of weight
$\lambda$ in $U_{r-1}(\liek)$. Moreover, since the sum in the first term of \eqref{u'}
is a vector in $U(\liek)\liey$ and $\dot X(U(\liek)\liey)\subset U(\liek)\liey$ for
$X \in \lieq^+$, it follows by hypothesis that $\dot X(u')\in U(\liek)\liey$
for every $X \in \lieq^+$. Also, since $u\in U(\liek)\liem^+$ and
$u\notin U(\liek)\liey$ the same facts hold for $u'$, therefore $u'\in\calS$.
This is a contradiction since $\text{deg}(u')<\text{deg}(u)$. Then $\calS = \emptyset$
and the proof of the theorem is completed.\qed
\end{proof}

\smallskip

\begin{cor}\label{q+dom} Let $u\in U(\liek)\liem^+$ be a $\lieq^+$-dominant vector
of weight $\lambda=a(\gamma_4+\delta)+b\gamma_3$ with $a,b\in\ZZ$. Then
$u\in U(\liek)\liey$.
\end{cor}

\smallskip

Next theorem will be used in an important way in Section \ref{mainproof}.
Its proof is similar to that of Theorem \ref{m+/n}. Consider the following subalgebra
of $\liek$,
\begin{equation}\label{qtilde}
\widetilde\lieq=\{X\in \liek: \dot X(V_\gamma^{\liek^+})=0 \,\,\,\text{for every}
\,\,\,\gamma \in\Gamma_1\}.
\end{equation}
It is easy to see that,
\begin{equation*}\label{qtilde1}
\widetilde\lieq = \liek^+\oplus \lieh_{\lier} \oplus
\langle\{X_{-\widetilde{\epsilon_3}+\widetilde{\epsilon_4}},
X_{-\widetilde{\epsilon_4}+\widetilde{\epsilon_1}},
X_{-\widetilde{\epsilon_3}+\widetilde{\epsilon_1}}\} \rangle,
\end{equation*}
where $\lieh_{\lier}$ is as in \eqref{hr}. Let $\widetilde Q$ denote the connected
Lie subgroup of $K$ with Lie algebra $\widetilde\lieq$.

\smallskip

If $Z\in (\liem^+)^\bot$ consider the linear map
$\widetilde{T}_Z:\widetilde\lieq \times(\liem^+)^\bot \rightarrow \liek$
given by
\begin{equation}\label{tildeTz}
\widetilde{T}_Z(X,Y)=[X,Z]+Y, \quad \quad X \in \widetilde\lieq
\quad \text{and}\quad Y\in (\liem^+)^\bot.
\end{equation}
Next proposition is the analogue of Proposition \ref{surj} and will be
used in the proof of Theorem \ref{k+dom}.

\smallskip

\begin{prop}\label{surj1} If $Z_o \in (\liem^+)^\bot$ is as in \eqref{Z_o} it follows that
\rm{Im}$(\widetilde{T}_{Z_o})=\liek$.
\end{prop}

\begin{proof} Using the definition of $\liey$ (see \eqref{n}) it is easy to see that,
\begin{equation}\label{liek}
\liek=\liey^\bot \oplus \langle \{X_{-\widetilde{\epsilon_2}-\widetilde{\epsilon_3}},
\,X_{-\widetilde{\epsilon_2}-\widetilde{\epsilon_4}},\,
X_{-\widetilde{\epsilon_3}-\widetilde{\epsilon_4}}\}\rangle.
\end{equation}
Now, since $\lieq \subset \widetilde\lieq$ it follows from Proposition \ref{surj} that, $$\widetilde{T}_{Z_o}\big(\lieq \times (\liem^+)^\bot\big) =
T_{Z_o}\big(\lieq \times (\liem^+)^\bot\big) = \liey^\bot.$$
Hence, it follows from \eqref{liek} that to complete the proof we need to show that $X_{-\widetilde{\epsilon_2}-\widetilde{\epsilon_3}}$,\,
$X_{-\widetilde{\epsilon_2}-\widetilde{\epsilon_4}}$, and \,$X_{-\widetilde{\epsilon_3}-\widetilde{\epsilon_4}}$
are in the image of $\widetilde{T}_{Z_o}$. In fact, a simple calculation shows that,
\begin{equation}
\widetilde{T}_{Z_o}\big(X_{-\widetilde{\epsilon_4}+\widetilde{\epsilon_1}},0\big)
\equiv a_1\, X_{-\widetilde{\epsilon_2}-\widetilde{\epsilon_4}},\quad \quad \quad \widetilde{T}_{Z_o}\big(X_{\gamma_1},0\big)\equiv a_2\, X_{-\widetilde{\epsilon_3}-
\widetilde{\epsilon_4}},
\end{equation}
and
\begin{equation}
\widetilde{T}_{Z_o}\big(X_{-\widetilde{\epsilon_3}+\widetilde{\epsilon_1}},0\big)\equiv
a_3\, X_{-\widetilde{\epsilon_2}-\widetilde{\epsilon_3}} +
a_4\, X_{-\widetilde{\epsilon_3}-\widetilde{\epsilon_4}},
\end{equation}
where, in all cases, the congruence is module the subspace $\liey^\bot$ and the constants
$a_i$ are nonzero for $1\leq i\leq 4$. This completes the proof.\qed
\end{proof}

\smallskip

\begin{thm}\label{k+dom} Let $u\in U(\liek)\liem^+$ be a $\liek^+$-dominant vector
of weight $\lambda=a(\gamma_4+\delta)+b\gamma_3$ with $a,b\in\NN_o$. Then
$u=0$.
\end{thm}

\begin{proof}Let $U(\liek)=\bigcup_{j\ge 0}U_j(\liek)$ and let $u\in U(\liek)\liem^+$
be a  $\liek^+$-dominant vector of weight $\lambda=a(\gamma_4+\delta)+b\gamma_3$
with $a,b\in\NN_o$. Assume that $u\neq 0$ and set $r=\text{deg}(u)$ (see \eqref{deg}).
Let $p_r:U_r(\liek)\rightarrow S_r(\liek^*)$
be the $K$-homomorphism defined in the proof of Theorem \ref{m+/n}. Observe that
$p_r(u)\ne 0$ because $u\notin U_{r-1}(\liek)$.

\smallskip

Since $u \in U(\liek)\liem^+$, and the $K$-homomorphism $p_r:U_r(\liek)\rightarrow
S_r(\liek^*)$ is defined via the Killing form of $\liek$, it follows that
\begin{equation}\label{p(u)Y}
p_r(u)(Y)=0 \quad \text{for every}\quad Y \in (\liem^+)^\bot.
\end{equation}

\smallskip

Also, since $u$ is a $\liek^+$-dominant vector of weight
$\lambda=a(\gamma_4+\delta)+b\gamma_3$, it follows from Proposition \ref{hw3} that $u\in
V_\gamma^{\liek^+}$ for $\gamma\in\Gamma_1$ with highest weight $\lambda$. Hence $\dot X(u)=0$ for every
$X\in\widetilde\lieq$. Then since $p_r$ is a $K$-homomorphism we have
\begin{equation}\label{X^k(p(u))1}
X^k(p_r(u))(Y)=p_r(\dot X^k(u))(Y)=0 \quad \text{for}\,\,\, X\in\widetilde\lieq,
\,\,\, Y\in \liek \,\,\,\text{and}\,\,\, k\in \NN.
\end{equation}
Now, since $\{\exp X: X\in \widetilde\lieq\}$ generates $\widetilde Q$,
it follows from \eqref{p(u)Y} and \eqref{X^k(p(u))1} that
\begin{equation*}\label{tildeQ}
p_r(u)(Ad(g)Y)=0 \quad \text{for}\quad g \in \widetilde Q \quad \text{and}\quad
Y\in (\liem^+)^\bot.
\end{equation*}
That is, $p_r(u)$ vanishes on the image of the map
$\widetilde\Phi:\widetilde Q \times (\liem^+)^\bot \rightarrow \liek$ defined by $\widetilde\Phi(g,Y)=Ad(g)Y$. Now, if $e\in \widetilde Q$ is the
identity element and $Z\in (\liem^+)^\bot$, then $(d\widetilde\Phi)_{(e,Z)} = \widetilde{T}_Z:\widetilde \lieq \times(\liem^+)^\bot \rightarrow \liek$.
Then it follows from Proposition \ref{surj1} that $(d\widetilde\Phi)_{(e,Z_o)}$
is  surjective. This implies that the image of $\widetilde\Phi$ contains an open
set of $\liek$, hence $p_r(u)=0$ as a polynomial function on $\liek$, which
is a contradiction. Therefore $u=0$ as we wanted to prove.\qed
\end{proof}

\medskip

Before stating the next results we define the following subalgebra of $\liek$,
\begin{equation}\label{s}
\lies=\liek^{-}\oplus\lieh_\liek\oplus\langle\{X_{\widetilde{\epsilon_3}+
\widetilde{\epsilon_1}}, X_{\widetilde{\epsilon_4}+\widetilde{\epsilon_1}}, T_{34},
X_1\}\rangle.
\end{equation}

The following result is the analogue of Proposition 4.9 of \cite {BCT}.
Although its proof uses the same idea as that of Proposition 4.9
we include it here because of some technical differences.

\begin{prop}\label{u+vE}
Let $u_0,u_1\in U(\liek)$ be such that $\dot X_1(u_0)=\dot
X_1(u_1)=0$. If $u_0+u_1E\equiv0\mod(U(\liek)\liey)$ then
$u_0\equiv u_1 \equiv0 \mod(U(\liek)\liey)$.
\end{prop}
\begin{proof}
Let $\lies$ be the subalgebra of $\liek$ defined in \eqref{s}. If
$\{S_1,\dots,S_t \}$ is an ordered basis of $\lies$, the
following is an ordered basis for $\liek$
\begin{equation}\label{base6} \{S_1,
\dots, S_t, T_{23}, T_{24}, X_{\delta}, X_{\psi_2}, X_{\delta_1},
X_{\psi_1}, X_{\delta_2}, X_4, X_3, X_2, S_{23}, S_{24}\},
\end{equation}
we refer the reader to \eqref{pd1}, \eqref{pd2}, \eqref{ST} and \eqref{psi}
for the notation.

\smallskip

Let $\calU_1$ (respectively $\calU_2$) be
the subspace of $U(\liek)$ spanned by those monomials that, when
written in the Poincar\'e-Birkhoff-Witt bases of $U(\liek)$
associated to \eqref{base6}, end with powers of $X_2$ (respectively
$S_{23}$) or before. Using that $\dot X_1(\lies)\subset \lies$ and
taking a close look at the action of $\dot X_1$ on the other
elements of the basis \eqref{base6} it follows that $\dot
X_1(\calU_1)\subset \calU_1$ and $\dot X_1(\calU_2)\subset \calU_2$.

\smallskip

Since $u_0+u_1E \in U(\liek)\liey$ in view of Lemma \ref{U(g)l1} we
can write
\begin{equation}\label{X0}
u_0+u_1E=a X_2+b S_{23}+c S_{24},
\end{equation}
with $a\in\calU_1$, $b\in\calU_2$ and  $c\in U(\liek)$.
Then applying $\dot X_1$ we obtain that,
\begin{equation}\label{X1}
u_1X_4=\dot X_1(a)X_2+aE+\dot X_1(b)S_{23}+\dot X_1(c)S_{24},
\end{equation}
and for every $k\ge2$ we get,
\begin{equation}\label{X11}
0=\dot X_1^k(a)X_2+k\dot X_1^{k-1}(a)E+\textstyle\binom{k}2 \dot
X_1^{k-2}(a)X_4+\dot X_1^k(b)S_{23}+\dot X_1^k(c)S_{24}.
\end{equation}

Now set $\widetilde \liey=\langle\{S_{23}, S_{24}\}\rangle$. If $n$
is sufficiently large so that $\dot X_1^n(a)=0$, using equation
\eqref{X11} and decreasing induction on $j$ it follows
that $\dot X_1^j(a)=0$ for every $0\le j\le n$. In particular,
$a=0$. Hence, from \eqref{X0} and \eqref{X1} we obtain that
$u_1X_4 \in U(\liek)\tilde\liey$ and $u_0+u_1E \in U(\liek)\tilde\liey$.
Now, using Lemma \ref{U(g)l2} and the fact that $\dot E(\liey)=0$ it follows
that $u_0\equiv u_1 \equiv0 \mod(U(\liek)\tilde\liey)$, therefore
$u_0\equiv u_1 \equiv0 \mod(U(\liek)\liey)$ as we wanted to prove.\qed
\end{proof}

\medskip

Next proposition will be used in Theorem \ref{sistLR} of
Section \ref{mainproof}.
\begin{prop}\label{parimpar} Let $\{\eta_j:j\in\NN_0\}$ be a sequence in
$U(\liek)$ such that $\eta_j\ne0$ for a finite number of $j'$s,
$\dot X_1(\eta_j)=0$ for every $j\in\NN_0$ and
$\sum_{j\ge0}\eta_jE^j\equiv0\mod(U(\liek)\liey)$. Then
$$\sum_{i\ge0}\eta_{2i}E^{2i}\equiv0\quad\textstyle{and}\quad\sum_{i\ge0}
\eta_{2i+1}E^{2i+1}\equiv0,$$
where the congruence is\!\!$\mod(U(\liek)\liey)$.
\end{prop}
\begin{proof} Let $\Delta=2X_4X_2-E^2$. Since $X_2$, $X_4$ and $E$ commute with
each other it follows that $(-1)^j\Delta^j\equiv E^{2j}$ mod$(U(\liek)\liey)$
for every $j\in \NN_0$. Also observe that $\dot X_1(\Delta)=0$. From now on
the proof follows in the same way as that of Proposition 4.11 of \cite {BCT},
simply changing the congruence mod$(U(\liek)X_2)$ for a congruence
mod$(U(\liek)\liey)$.\qed
\end{proof}

\section{An estimate on the Kostant degree}\label{Kostant}

In this section we introduce the \textit{degree property} and show
that every $b\in P(U(\lieg)^K)$ has the degree property. This result
is used in Proposition \ref{Btilde1} bellow. We also show that to
prove Theorem \ref{main'}, and therefore our main result Theorem
\ref{main}, it is enough to prove Theorem \ref{Btilde2} bellow.

\begin{defn}\label{degprop}
Let $b=b_m\otimes Z^m+\cdots+b_0\in U(\liek)^M\otimes U(\liea)$ with
$b_m \neq 0.$ We say that $b$ has the \textit{degree property} if
$d(b_{m-j})\le m+2j$ for every $0\le j\le m$.
\end{defn}

We begin by recalling a few facts about $\lies$-triples in $\lieg$.
Recall that an $\lies$-triple is a set of three linearly independent
elements $\{x, e,f\}$ in $\lieg$ such that $[x,e]=2e$, $[x,f]=-2f$
and $[e,f]=x$. The $\lies$-triple $\{x,e,f\}$  is called \textit{normal}
if $e, f \in\liep$ and $x\in\liek$. A normal $\lies$-triple $\{x,e,f\}$
is called \textit{principal} if $e$ (and hence $f$) is a regular element in
$\liep$. Theorem 3 of \cite{KR} guarantee that principal normal
$\lies$-triples exist, and in Theorem 6 of the same paper it is
proved that any two principal normal $\lies$-triples are
$K_\theta$-conjugate, where $K_\theta$ is the subgroup of all
elements in $G$ that commute with $\theta$.

\smallskip

Fix a principal normal $\lies$-triple $\{x,e,f\}$ in $\lieg$ and set
$z=x/2$. In Proposition 1 of \cite{Ti1} it is proved that the map
$ad(z) : \liep \rightarrow \liep$ is diagonalizable with eigenvalues
$1$, $-1$ and possibly $0$. Since in our case $\lieg \simeq \lief_4 $, the
eigenvalues of $ad(z)$ in $\lieg$ are $-2$, $-1$, $0$, $1$ and $2$
(see the proof of Proposition 1 of \cite{Ti1}), then the next result
follows.

\begin{lem}\label{z1} The map  $ad(z) : \liek \rightarrow \liek$ is
diagonalizable and its highest eigenvalue is 2.
\end{lem}

In Corollary 9 of \cite{Ti1} it is shown that if $\lieg_o$ is a
semisimple Lie algebra over $\RR$, different from $\liesl(2,\RR)$,
and $V_\gamma$ is an irreducible $K$-module of type $\gamma \in
\Gamma$ then $d(\gamma)$ is the highest eigenvalue of $z$ in
$V_\gamma$. From this result the following lemma follows.

\begin{lem}\label{z2} Let $V$ be a finite dimensional $K$-module and let $n$ be the
highest eigenvalue of $z$ in $V$. If $u \in V^M$ and $u \ne 0$, then
$d(u)\le n$.
\end{lem}

As an application of Lemma \ref{z1} and Lemma \ref{z2} we obtain the
following result that will be useful in what follows.

\begin{lem}\label{z3} If $u \in U_m(\liek)^M$ and $u \ne 0$, then
$d(u)\le 2m$.
\end{lem}

Recall that $P:U(\lieg)\longrightarrow U(\liek)\otimes U(\liea)$ is
the projection on the first summand of the direct sum
$U(\lieg)=\bigl(U(\liek)\otimes U(\liea)\bigr)\oplus U(\lieg)\lien$,
associated to an Iwasawa decomposition
$\lieg=\liek\oplus\liea\oplus\lien$ adapted to $\liek$. The proof of
the following result follows easily by choosing an appropriate
Poincar\'e-Birkhoff-Witt bases of $U(\lieg)$.

\begin{lem}\label{P(U_n)}  $P( U_m(\lieg)) = \displaystyle{\sum_{\substack{0\le \ell\le m}}}
U_{m-\ell}(\liek)\otimes Z^\ell$ for every $m\ge 0$.
\end{lem}

Let $\sigma: S(\lieg)\longrightarrow U(\lieg)$ be the symmetrization
mapping. It is known that $\sigma$ is a $K$-linear isomorphism. Let
$\varphi: U(\liek)\otimes S(\liep)\longrightarrow U(\lieg)$ be the
$K$-linear isomorphism defined by $\varphi(u \otimes
p)=u\,\sigma(p)$. Then we have,

\begin{equation*}
U(\lieg)^K = \displaystyle \sum_{m\ge 0}
\big(U(\liek)\,\sigma\!\left(S_m(\liep)\right)\big)^K.
\end{equation*}

\begin{thm}\label{degP(u)}Let $u\in \big(U(\liek)\sigma\!\left(S_m(\liep)\right)
\big)^K$ where $m$ is the smallest possible. Then $P(u)= b_m\otimes
Z^m+\cdots+b_0 \in U(\liek)^M\otimes U(\liea)$, $b_m\neq 0$ and
$d(b_{m-j})\le m+2j$ for $0\le j\le m$.
\end{thm}

\begin{proof}Let $\widetilde u \in \big(U(\liek)\otimes S_m(\liep)\big)^K$ be
such that $\varphi(\widetilde u)=u$. Write $S_m(\liep)=\textstyle{\sum W_{\tau}}$
where the sum runs over a finite set $J \subset \Gamma$. Then by Schur's Lemma
we have,
\begin{equation}\label{schur}
\big(U(\liek)\otimes S_m(\liep)\big)^K = \displaystyle \sum_{\tau \in J}
\big(U(\liek)_{\tau^*}\otimes W_\tau \big)^K,
\end{equation}
\noindent where $\tau^*$ is the contragredient representation of $\tau$, and
$U(\liek)_{\tau^*}$ denotes the $\tau^*$-isotypic component of $U(\liek)$.

Let $\lieq$ be a subspace of $\liep$ such that $\liep=\liea \oplus
\lieq$ and let $\{X_1, \dots, X_r\}$ be an ordered bases of
$\lieq$\,. If $a=(a_1,\dots,a_r)$ with $a_i \in \NN_0$, and
$X^a=X_1^{a_1}\cdots X_r^{a_r}$ in $S(\liep)$, it follows that
$\{Z^{\ell} X^a: 0\le \ell+|a| \le m \}$ is a bases of $S_m(\liep)$,
where $|a|= a_1+\cdots+a_r$. Then, in view of \eqref{schur}, we can
write
\begin{equation*}
\tilde u = \displaystyle \sum_{0\le \ell+|a| \le m} u_{\ell,
a}\otimes Z^{\ell} X^a,
\end{equation*}
\noindent where $u_{\ell, a}$ belongs to the $K$-module $V=
\textstyle{\sum_{\tau \in J} U(\liek)_{\tau^*}^M}$ for
every pair $(\ell, a)$. Then,
\begin{equation}\label{P(u)}
P(u) = \displaystyle \sum_{0\le \ell+|a| \le m} P\big(u_{\ell,
a}\,\sigma( Z^{\ell} X^a)\big)= \displaystyle \sum_{0\le \ell+|a|
\le m} u_{\ell, a}\,P\big(\sigma( Z^{\ell} X^a)\big).
\end{equation}

Now, since $\sigma( Z^{\ell} X^a)\in U_{\ell+|a|}(\lieg)$, it
follows from Lemma \ref{P(U_n)} that
$$P(\sigma (Z^{\ell} X^a))= \displaystyle \sum_{0\le j \le \ell+|a|} v_{\ell, a, j}
\otimes Z^j,$$ \noindent with $v_{\ell, a, j}\in
U_{\ell+|a|-j}(\liek)$. Hence from \eqref{P(u)} we have,
$$P(u)= \displaystyle \sum_{0\le j \le m} \big(\sum_{j\le \ell+|a| \le m}u_{\ell,
a}\,v_{\ell, a, j}\big)\otimes Z^j. $$ Then from the uniqueness of
the coefficients $b_j$ it follows that,
\begin{equation}\label{bj'}
b_j = \displaystyle \sum_{j\le \ell+|a| \le m}u_{\ell, a}\,v_{\ell,
a, j} \quad\quad \text{for} \quad\quad 0\le j\le m,
\end{equation}
where $v_{\ell, a, j}\in U_{\ell+|a|-j}(\liek)\subset U_{m-j}(\liek)$
for every pair $(\ell,a)$. Hence from \eqref{bj'} we obtain that,
\begin{equation}\label{bj''}
b_j \in \langle V \cdot  U_{m-j}(\liek)\rangle^M \subset U(\liek)^M
\quad\quad \text{for} \quad\quad 0\le j\le m.
\end{equation}
Recall that $\langle S \rangle$ denotes the linear space spanned by
the set $S$. Observe that in this case
$\langle V \cdot  U_{m-j}(\liek)\rangle$ is a $K$-module.

Now, since the highest eigenvalue of $z$ in $\liep$ is $1$, it
follows that the highest eigenvalue of $z$ in $S_m(\liep)$ is $m$.
Then $d(\tau)\le m$ for every $\tau \in J$, and therefore
$d(\tau^*)\le m$ for every $\tau \in J$. This implies that the
highest eigenvalue of $z$ in $V$ is less or equal to $m$. On the
other hand, we know that the highest eigenvalue of $z$ in
$U_{m-j}(\liek)$ is less or equal to $2(m-j)$, hence the highest
eigenvalue of $z$ in $\langle V \cdot  U_{m-j}(\liek)\rangle$ is
less or equal to $m+2(m-j)$. Then, from Lemma \ref{z2} and
\eqref{bj''} it follows that $d(b_{j})\le m+2(m-j)$ for $0\le j\le
m$, and therefore $d(b_{m-j})\le m+2j$ for $0\le j\le m$, as we
wanted to prove.\qed
\end{proof}

\begin{thm}\label{degP(u)'}Let $b\in P(U(\lieg)^K)$ be such that $b= b_m\otimes
Z^m+\cdots+b_0 $ with $b_m\neq 0$, then $d(b_{m-j})\le m+2j$ for
every $0\le j\le m$.
\end{thm}

\begin{proof} Let $u \in U(\lieg)^K$ be such that $P(u)=b$. Since $b_m\neq 0$,
it follows from Corollary 7.3 of \cite{KT} that $u\in
\big(U(\liek)\sigma\left(S_m(\liep)\right) \big)^K$ and $m$ is the
smallest possible. Hence the result follows from Theorem
\ref{degP(u)}.\qed
\end{proof}

\medskip

Our next goal is to show that Theorem \ref{main'} follows from Theorem \ref{Btilde2}
bellow. In the next lemma we single out a particular
element $\omega \in B$. This element is a scalar
multiple of $P(\Omega)$, where $\Omega$ is the Casimir of $\lieg$.

\begin{lem}\label{casimir}
There exist $\omega = \omega_2\otimes Z^2+\omega_1\otimes
Z+\omega_0\in P(U(\lieg)^{K})\subset B$ such that $\omega_2=1$, $\omega_1$ is
a nonzero scalar, $\omega_0$ is a scalar multiple of the Casimir element of
$\liem$ and $d(\omega_0)\le 4$.
\end{lem}

\begin{prop}\label{degprop1}For any $b\in U(\liek)^M\otimes
U(\liea)$ there exist $n\in \NN_0$ such that $b\omega^n$ has the
degree property.
\end{prop}

\begin{proof}Let $b=b_m\otimes Z^m+\cdots+b_0\in U(\liek)^M\otimes
U(\liea)$. Fix $n \in \NN_0$ sufficiently large so that
$d(b_{m-j})\le m+2n+2j$\, for every $0\le j\le m$. A simple
calculation shows that
\begin{equation}\label{w^n}
\omega^n = \sum_{k=0}^{2n} \widetilde\omega_{k,\,n}\otimes Z^{2n-k},
\end{equation}
where $\widetilde\omega_{k,\,n} =
\displaystyle\sum_{i=0}^{[k/2]}\textstyle\binom{n}{k-i}\textstyle\binom{k-i}{i}
\,\omega_1^{k-2i}\omega_0^i$ for $0\le k\le 2n$, and that
\begin{equation}\label{bw^n}
b\omega^n = \sum_{j=0}^{m+2n}\big(\sum_{k=0}^{\text{min}\{j,\, 2n\}}
b_{m+k-j}\,\widetilde\omega_{k,\,n}\,\big)\otimes Z^{m+2n-j}.
\end{equation}
Then if $(b\omega^n)_\ell$ denotes the coefficient of $Z^\ell$ in
$b\omega^n$, we have
\begin{equation*}
\begin{split}
d((b\omega^n)_{m+2n-j})
&\le \text{max}\{d(b_{m+k-j}\,\widetilde \omega_{k,\,n}): 0\le k \le j\}\\
&= \text{max}\{d(b_{m+k-j}) + d(\widetilde \omega_{k,\,n}): 0\le k \le j\}\\
&\le \text{max}\{m+2n+2(j-k)+2k: 0\le k \le j\}\\
&=m+2n+2j,
\end{split}
\end{equation*}
for every $0\le j\le m+2n$. Hence $b\omega^n$ has the degree
property.\qed
\end{proof}

\smallskip

It is now convenient to introduce the following notation, for any $m
\in \NN_0$ and $0\le r\le m$ define $d_r$ as follows,
\begin{equation}\label{dr}
d_r=\left[\frac{3m-2r+2}2\right].
\end{equation}
In the next lemma we obtain an upper bound on the Kostant degree of
the coefficients $b_r$ of certain $b\in U(\liek)^M \otimes
U(\liea)$.

\smallskip

\begin{lem}\label{degprop2} Let $b=b_m\otimes Z^m+\cdots+b_0\in U(\liek)^M\otimes
U(\liea)$ with $b_m \ne 0$. If $b\omega$ has the degree property
then $d(b_r)\le 2d_r$ for every $0\le r\le m$.
\end{lem}

\begin{proof} Let $(b\omega)_\ell$ denote the coefficient of $Z^\ell$ in
$b\omega$. It follows from \eqref{w^n} and \eqref{bw^n}, or directly by
computing $b\omega$, that
\begin{equation}\label{bw}
b_{m-j}= (b\omega)_{m+2-j}- b_{m-j+1}\,\omega_1-b_{m-j+2}\,\omega_0
\end{equation}
for $0\le j\le m$, with the understanding that $b_{m+1}=b_{m+2}=0$.
Then, since $\omega_1$ is a scalar and $d(\omega_0)\le 4$, from
\eqref{bw} we obtain that
\begin{equation}\label{bw1}
d(b_{m-j})\le \text{max}\{ d((b\omega)_{m+2-j}), d(b_{m-j+1}),
d(b_{m-j+2})+ 4\}.
\end{equation}
Hence, using \eqref{bw1} and the fact that $b\omega$ has the degree
property, it follows by induction on $j$ that $d(b_{m-j})\le m+2+2j$
for every $0\le j\le m$. Now, since the Kostant degree of any element of
$U(\liek)^M$ is even (see (ii) and (iii) of Proposition \ref{hw3}),
it follows that $d(b_r)\le 2d_r$ for every $0\le r\le m$.\qed
\end{proof}

\medskip

Let $b=b_m\otimes Z^m+\cdots+b_0\in B$ be such that $d(b_r)\le 2d_r$
for $0\le r \le m$, where $d_r$ is as in \eqref{dr}. Using Proposition
\ref{hw3} and the above bound on $d(b_r)$ we can decompose the coefficients
$b_r$ of $b$ as follows,
\begin{equation}\label{types}
b_r=\sum_{t=0}^{2d_r}\,\, \sum_{\max\{0,\,t-d_r\}\le
i\le[t/2]}b^r_{2i,\,t-2i}  \quad \quad \text{for} \quad \quad 0\le r\le m,
\end{equation}
where $b^r_{2i,t-2i}$ is the component of $b_r$ in the isotypic component
of $U(\liek)^M$ of type $(2i,t-2i)$. Consider now the following linear
subspace of $B$,
\begin{equation}\label{Btilde}
\widetilde B=\{b\in B:\, b^{2k}_{2i,j}=0 \text{ if } i+j\le k \text{
and } 0\le 2k\le\deg(b)\}.
\end{equation}
That is, $\widetilde B$ consists of the elements $b\in B$ such that
the $K$-types $b^{2k}_{2i,j}$ that occur in the coefficient $b_{2k}$
of $b$, have Kostant degree greater than $2k$ for all $k$ such that
$0\le 2k \le \deg(b)$.

\begin{prop}\label{Btilde1}
Let $b=b_m\otimes Z^m+\cdots+b_0\in B$, $b_m\ne0$, and $d(b_r)\le
2d_r$ for\, $0\le r \le m$. Then there exist $\widetilde b \in
\widetilde B$ such that $d(\widetilde b_r)\le 2d_r$ for $0\le r \le
m$, $\widetilde b_m=b_m$ if $m$ is odd, and $d(b_m-\, \widetilde
b_m)\le m$ if $m$ is even. Moreover $\widetilde
b^{r}_{2i,j}=b^{r}_{2i,j}$ if $i+j=d_r$ for every $0\le r \le m$.
\end{prop}

\begin{proof} Let $b=b_m\otimes Z^m+\cdots+b_0\in B$ be such that
$b_m\ne0$ and $d(b_r)\le 2d_r$ for $0\le r \le m$. Set $p=2[m/2]$ and
using \eqref{types} define,
\begin{equation*}
c_p=\sum_{t=0}^p \,\,\sum_{\max\{0,\,t-\tfrac {p}{2}\}\le i\le [t/2]}b^p_{2i,\,t-2i}.
\end{equation*}
That is, $c_p$ contains all the $K$-types of $b_p$ of Kostant degree
smaller or equal to $p$. Hence, $c_p\in U(\liek)^M$ and $d(c_p)\le
p$. Since $p$ is even $c_p\otimes Z^p\in\big(U(\liek)^M\otimes
U(\liea)\big)^W$. Then from Proposition \ref{cor9} it
follows that $c_p\otimes Z^p$ is the leading term of an element
$c^{(p)}=c_p\otimes Z^p+\cdots\in P(\U(\lieg)^K).$ Now define
$b^{(p)}=b-c^{(p)}\in B$. All the $K$-types that occur in the
$p$-coefficient of $b^{(p)}$ have Kostant degree greater  than $p$
and, since $c^{(p)}\in P(\U(\lieg)^K)$, it follows from Theorem
\ref{degP(u)'} that $d(b^{(p)}_r)\le 2d_r$ for\, $0\le r \le m$.
Moreover, the $K$-types of Kostant degree $2d_r$ of $b^{(p)}_r$ and
$b_r$ are the same for $0\le r \le m$.

Considering now the $(p-2)$-coefficient of $b^{(p)}$ we construct,
in a similar way, elements $c^{(p-2)}\in P(\U(\lieg)^K)$ and
$b^{(p-2)}=b^{(p)}-c^{(p-2)}\in B$, such that the coefficients of
$b^{(p-2)}$ corresponding to degrees greater than $p-2$ are the same
as those of $b^{(p)}$, and all the $K$-types that occur in the
$(p-2)$-coefficient of $b^{(p-2)}$ have Kostant degree greater than
$p-2$. Moreover, since $c^{(p-2)}\in P(\U(\lieg)^K)$, Theorem
\ref{degP(u)'} implies that $d(b^{(p-2)}_r)\le 2d_r$ for\, $0\le r
\le m$, and that the $K$-types of Kostant degree $2d_r$ of $b^{(p-2)}_r$
and $b_r$ are the same for every $0\le r \le m$.

Continuing in this way we obtain a sequence
$b^{(p)},b^{(p-2)},\dots,b^{(0)}$ of elements in $B$ of degree at
most $m$. If we set $\widetilde b=b^{(0)}$ it is clear that
$\widetilde b \in \widetilde B$ and that $\widetilde b$ has all the
required properties.\qed
\end{proof}

\smallskip

Finally, in Proposition \ref{Btilde4} below, we show that next
theorem implies Theorem \ref{main'} (and therefore Theorem \ref{main}).
The proof of Theorem \ref{Btilde2} will be done in the next section.

\begin{thm}\label{Btilde2}
Let $b=b_m\otimes Z^m+\cdots+b_0 \in \widetilde B$ be such that
$d(b_r)\le 2d_r$ for every $0\le r \le m$. Then $b=0$.
\end{thm}

If we assume that Theorem \ref{Btilde2} holds we obtain the
following corollary.

\begin{cor}\label{Btilde3}
Let $b=b_m\otimes Z^m+\cdots+b_0 \in B$ be such that $b_m \ne 0$ and
$b\omega$ has the degree property. Then $m$ is even and $b$ has the
degree property.
\end{cor}
\begin{proof} Since $b\omega$ has the degree property, it follows
from Lemma \ref{degprop2} that $d(b_r) \le 2d_r$ for $0\le r \le m$.
Then Proposition \ref{Btilde1} implies that there exist $\widetilde
b \in \widetilde B$ such that $d(\widetilde b_r)\le 2d_r$ for $0\le
r \le m$, $\widetilde b_m=b_m$ if $m$ is odd and $\widetilde
b^{r}_{2i,j}=b^{r}_{2i,j}$ if $i+j=d_r$ for $0\le r \le m$. On the
other hand, Theorem \ref{Btilde2} implies that $\widetilde b=0$.
Hence, $m$ must be even and $b^{r}_{2i,j}=0$ if $i+j=d_r$ for $0\le
r \le m$, which implies that $b$ has the degree property.\qed
\end{proof}

\begin{prop}\label{Btilde4}Let $b=b_m\otimes Z^m+\cdots+b_0 \in B$
with $b_m \ne 0$. Then, $m$ is even and $b$ has the degree property.
In particular $d(b_m)\le m$, and therefore Theorem \ref{main'} holds.
\end{prop}
\begin{proof}
Let $b=b_m\otimes Z^m+\cdots+b_0 \in B$ be as in the
statement of the theorem. It follows from Proposition \ref{degprop1}
that there exist $n\in \NN_0$ such that $b\omega^n$ has the degree
property. Now, since $b\omega^{n-1}= b_m \otimes Z^{m+2(n-1)}+ \dots
\in B$ and $b_m \ne 0$, it follows from Corollary \ref{Btilde3} that
$m+2(n-1)$ is even and $b\omega^{n-1}$ has the degree property.
Hence $m$ is even, and from Corollary \ref{Btilde3} and induction on
$k$ it follows that $b\omega^{n-k}$ has the degree property for
every $0\le k\le n$. In particular $b$ has the degree property, as
we wanted to prove.\qed
\end{proof}

\section{Proof of Theorem \ref{Btilde2}}\label{mainproof}

Our goal in this section is to prove Theorem \ref{Btilde2}. To do this,
given any $b=b_m\otimes Z^m+\cdots+b_0\in B$ such that $d(b_r)\le 2d_r$
for $0\le r\le m$, we will construct a linear system of equations in $U(\liek)$
where the unknowns are $\liek^+$-dominant vectors associated to certain
$K$-types of the coefficients of $b$ (see Theorem \ref{sistLR}).
This system will allow us to carry out a decreasing induction process that,
when applied to $b\in\widetilde B$, will lead to the proof of
Theorem \ref{Btilde2}.

\smallskip

Let $b=b_m\otimes Z^m+\cdots+b_0\in B$ be such that $d(b_r)\le 2d_r$
for $0\le r\le m$. As indicated in \eqref{types} we can decompose the
coefficient $b_r$ of $b$ as follows,
\begin{equation}\label{types1}
b_r=\sum_{t=0}^{2d_r}\, \sum_{\max\{0,\,t-d_r\}\le
i\le[t/2]}b^r_{2i,t-2i}  \quad \quad \text{for} \quad \quad 0\le r\le m.
\end{equation}
We find it very convenient to keep in mind the following array of the $K$-types
that occur in $b_r$.
\begin{align}\label{array}
\begin{matrix} b_r= b^r_{2d_r,0}&+\;\;\; b^r_{2d_r-2,1}&+\;\;\;
b^r_{2d_r-4,2}&+\;\;\; b^r_{2d_r-6,3}&+\;\;\; \cdot & \cdot & \cdot
&
+\;\;\; b^r_{0,d_r}\;\;\;\\
                                  &+\;\;\; b^r_{2d_r-2,0}&+\;\;\;
b^r_{2d_r-4,1}&+\;\;\; b^r_{2d_r-6,2}&+\;\;\; \cdot & \cdot & \cdot
& +\;\;\;
b^r_{0,d_r-1}\\
                                  &                      &+\;\;\;
b^r_{2d_r-4,0}&+\;\;\; b^r_{2d_r-6,1}&+\;\;\; \cdot & \cdot & \cdot
&
+\;\;\; b^r_{0,d_r-2}\\
                                  &                      &
     &+\;\;\; b^r_{2d_r-6,0}&+\;\;\; \cdot & \cdot & \cdot & +\;\;\;
b^r_{0,d_r-3}\\
                                  &                      &
     &                      &\;\;\;\;\;\;  \cdot & \cdot & \cdot
&  \;\;\;\cdot \\
                                  &                      &
     &                      &                    & \cdot      & \cdot
&  \;\;\;\cdot \\
                                  &                      &
     &                      &                    &            & \cdot &
\;\;\;\cdot\\
                                  &                      &
     &                      &              &       &       & +\;\;\;\;
b^r_{0,0}\;\;\;.
\end{matrix}
\end{align}

Observe that the parameter $t$ used in (\ref{types1})
may be regarded as a label for the skew diagonals of the array
(\ref{array}). In fact, for $0\le t\le 2d_r$ we shall refer to the
set $\{b^r_{2i,t-2i}:\max\{0,t-d_r\}\le i\le[t/2]\}$ as the skew
diagonal associated to $t$. Also observe that the Kostant degree is
constant along the rows of the array \eqref{array}, it takes the
values $2d_r,2d_r-2, \dots ,0$ from the top to the bottom row of the
array corresponding to $b_r$.

\smallskip

Let $T\in \NN_o$ denote the label of the skew diagonals in the array
corresponding to $b_0$. We will use $T$ as a parameter for a decreasing
induction. For $m\le T\le 2d_0$ if $m$ is even, and $m-1\le T\le 2d_0$ if
$m$ is odd, consider the following propositional function
associated to $b$,
\begin{equation}\label{Pfunction}
P(T): b_r=\sum_{t=0}^{\min\{T-r,2d_r\}}\sum_{\max\{0,t-d_r\}\le
i\le[t/2]} b^r_{2i,t-2i}, \quad\quad 0\le r\le m.
\end{equation}
\noindent Observe that $P(T)$ holds if and only if $b^r_{2i,t-2i}=0$ for
$t>\min\{T-r,2d_r\}$ for every $0\le r\le m$. Also, in view of \eqref{types},
it follows that $P(2d_0)$ holds. This will be the starting point of our inductive
argument.

\smallskip

Let $E$, $X_{\delta}$, $H$, $Y$ and $\widetilde Y$ be as in
Section \ref{F4}. Recall that $\dot E(H)=-\frac12E$,
$\dot X_{\delta}(\widetilde Y)=X_{\delta}$ and $\dot X_{\delta}(H)=
\dot E(X_{\delta})=0$. In the following lemma we state some properties
of the derivations $\dot E$ and $\dot X_{\delta}$, we refer to Lemma 6.1
of \cite {BCT} for their proof.

\begin{lem}\label{Lemma 4}

(i) $\dot E^k(H^k)=k!(-{\frac12}E)^k$ and $\dot E^k(H^j)=0$ if $k>j$.

\noindent (ii) $\dot E^k\varphi_k(H)=(-{\frac12}E)^k$, where $\varphi_k$ is as in
\eqref{phi}.

\noindent (iii) $\dot X_{\delta}^k((-\widetilde Y)^k)=k!(-X_{\delta})^k$
and $\dot X_{\delta}^k((-\widetilde Y)^j)=0$ if $k>j$.

\noindent (iv) $\dot X_{\delta}^k\varphi_k(a-\widetilde Y)=(-X_{\delta})^k$ for any
$a\in \CC$.
\end{lem}

The following proposition is the analogue of Proposition 6.2 of \cite{BCT}.
Its proof is the same as that of Proposition 6.2 and it is obtained by
applying $\dot X_{\delta}^{T-n-\ell}$ to $\epsilon(\ell,n)$ of
Theorem \ref{thm2}, and using Lemma \ref{cb} and Lemma \ref{Lemma 4}.
Also observe that the derivation $\dot X_{\delta}$ preserves
the ideal $U(\liek)\liem^+$.
\begin{prop}\label{sist1}
Let $b=b_m\otimes Z^m+\cdots+b_0\in B$ be such that $d(b_r)\le 2d_r$
for $0\le r\le m$, and assume that $P(T)$ holds for $m\le T\le 2d_0$.
Then for every $(\ell,n)$ such that $0\le\ell,n$ and $\ell+n\le T$ we have
\begin{equation}\label{Sumas12}
(-1)^n\Sigma_1E^n-(-1)^\ell\Sigma_2E^\ell\equiv0\quad\mod(U(\liek)\liem^+),
\end{equation}
\noindent where
\begin{align}\label{sigma12}
\begin{split}
\Sigma_1&=\sum_{(i,r)\in I_1}A_{i,r}(T,n,\ell)\,\dot X_{\delta}^{T-\ell-i}
\dot E ^{\ell+i-r}(b_r)\,E^{r-i}X_{\delta}^{i-n},\\
\Sigma_2&=\sum_{(i,r)\in I_2}A_{i,r}(T,\ell,n)\,\dot X_{\delta}^{T-n-i}
\dot E^ {n+i-r}(b_r)\,E^{r-i}X_{\delta}^{i-\ell},\\
\end{split}
\end{align}
and
$$A_{i,r}(T,n,\ell)=(-{\textstyle\frac12)}^{r-i}(-1)^{i-n}r!{\textstyle
\binom{T-n-\ell}{i-n}\binom \ell{r-i}},$$
$$I_1=\{(i,r)\in N_0^2: n\leq i\leq\min\{m,T-\ell\}, i\leq
r\leq\min\{m,i+\ell\}\},$$
$$I_2=\{(i,r)\in N_0^2: \ell\leq i\leq\min\{m,T-n\}, i\leq
r\leq\min\{m,i+n\}\}.$$
\end{prop}

\smallskip

Next proposition is the analogue of Proposition 6.3 of \cite{BCT} and
its proof is the same as that of Proposition 6.3. It is obtained by
replacing $b_r$ in \eqref{sigma12} by its decomposition in $K$-types given
in \eqref{Pfunction}, then one uses Proposition \ref{hw3} (iv) to simplify
the sums $\Sigma_1$ and $\Sigma_2$, and finally one multiply both sums on
the right by $X_{\delta}^T$ and then change in each term a certain number of
$X_{\delta}$'s by the same number of $X_4$'s so that $\Sigma_1$ and
$\Sigma_2$ become weight vectors with respect to $\lieh_\liek$. Here we
use that $X_{\delta} \equiv X_4\mod(U(\liek)\liem^+)$ and that the derivation
$\dot X_{\delta}$ preserves the ideal $U(\liek)\liem^+$.

\begin{prop}\label{sist2}
Let $b=b_m\otimes Z^m+\cdots+b_0\in B$ be such that $d(b_r)\le 2d_r$
for $0\le r\le m$, and assume that $P(T)$ holds for $m\le T\le 2d_0$.
Then for every $(\ell,n)$ such that $0\le\ell,n$ and $\ell+n\le T$ we have
\begin{equation}\label{Sumas123}
(-1)^n\Sigma_1E^n-(-1)^\ell\Sigma_2E^\ell\equiv0\quad\mod(U(\liek)\liem^+),
\end{equation}
\noindent where
\begin{equation*}
\begin{split}
\Sigma_1=\sum_{\substack{(i,r)\in I_1\\ \max\{0,T-r-d_r\}\leq k
\leq[{\textstyle\frac{T-r}2}]}}&A_{i,r}(T,n,\ell)\,\dot
X_{\delta}^{T-\ell-i}\dot E^{\ell+i-r}(b^r_{2k,T-r-2k})\\
           &\times E^{r-i}X_{\delta}^{T-k}X_4^{k+i-n},\\
\Sigma_2=\sum_{\substack{(i,r)\in I_2\\ \max\{0,T-r-d_r\}\leq k
\leq[{\textstyle\frac{T-r}2}]}}&A_{i,r}(T,\ell,n)\,\dot
X_{\delta}^{T-n-i}\dot E^{n+i-r}(b^r_{2k,T-r-2k})\\
           &\times E^{r-i}X_{\delta}^{T-k}X_4^{k+i-\ell},
\end{split}
\end{equation*}
with the understanding that the $K$-types $b^r_{2k,T-r-2k}$ that do
not occur in $b_r$ are assumed to be zero. Moreover, in equation
(\ref{Sumas123}) all the terms of the left hand side are weight
vectors of weight $(2T-\ell-n)\gamma_1+T(\gamma_2+\delta)$.
\end{prop}

\smallskip

The equations (\ref{Sumas123}) may be regarded as a system of linear
equations where the unknowns, $\dot X_{\delta}^{T-j-i}
\dot E^{j+i-r}(b^r_{2k,T-r-2k})$,
are derivatives of the $K$-types that occur in the $T-r$ skew
diagonal of the coefficient $b_r$ of $b$ (see (\ref{array})).
Since the unknowns in this system are, in general, not
$\liek^+$-dominant we are going to replace the system by an
equivalent one where all the unknowns become $\liek^+$-dominant
vectors associated to the $K$-types $b^r_{2k,T-r-2k}$.

Let $\widetilde\epsilon(\ell,n)$ be the left hand side of
equation (\ref{Sumas123}). For $0\le n\le\min\{2d_m,T\}$
and $0\le L\le\min\{2m,T\}-n$ consider the following linear
combination,
\begin{equation}\label{epsiL}
\mathcal E_L(n)=\sum_{\ell=0}^L(-2)^\ell\textstyle \binom
L\ell\widetilde\epsilon(\ell,n)E^{L-\ell}X_4^{\ell+n}.
\end{equation}
Under the hypothesis of Proposition \ref{sist2} we
have $\mathcal E_L(n)\equiv0\mod(U(\liek)\liem^+)$. Also set,
\begin{equation*}
\mathcal E_L^1(n)=\sum_{\ell=0}^L(-2)^\ell\textstyle \binom
L\ell\Sigma_1E^{L-\ell}X_4^{\ell+n}
\quad \text{and}\quad \mathcal E_L^2(n)=\sum_{\ell=0}^L2^\ell\textstyle \binom
L\ell\Sigma_2X_4^{\ell+n}.
\end{equation*}
Then it follows that
\begin{equation}\label{epsiL'}
\mathcal E_L(n)=(-1)^n\mathcal E_L^1(n)E^n-\mathcal E_L^2(n)E^L.
\end{equation}

The following lemma is the analogue of Lemma 6.5 of \cite{BCT}. For the
symplectic group Sp($n$,1) the vectors $D_k(b_{2i,j})$
are $\liek^+$-dominant, however in F$_4$ this property does not hold.

\begin{lem}\label{LemDk} Let $b_{2i,j}\in U(\liek)^M$ be an
$M$-invariant element of type $(2i,j).$ For $0\leq k\leq 2i$ define,
\begin{equation}\label{Dk}
D_k(b_{2i,j})=\Sigma_{\ell=0}^k(-2)^\ell{\textstyle\binom k\ell
\binom{j+\ell}\ell^{-1}}\dot X_{\delta}^{2i-\ell}\dot
E^{j+\ell}(b_{2i,j})\,E^{k-\ell}X_4^\ell.
\end{equation}
Then $D_k(b_{2i,j})$ is a vector of weight
$i(\gamma_4+\delta)+(j+k)\gamma_3$ with respect to $\lieh_\liek$,
$\dot X(D_k(b_{2i,j})) \equiv 0\mod\!(U(\liek)\liey)$  for every
$X \in \lieq^+$ and $\dot X_1(D_k(b_{2i,j}))=0$.
\end{lem}

\begin{proof} Recall that $\lieq^+$ is the linear span of
$\{X_{\alpha}:\alpha\in\Delta^+(\liek,\lieh_\liek)-\{\gamma_1\}\}$.
Since $\gamma_1$ is a simple root in $\Delta^+(\liek,\lieh_\liek)$,
if $\alpha$ is a positive root it follows that $\alpha-\gamma_1$ is
either a positive root different from $\gamma_1$ or it is not a root.
Hence if $u \in U(\liek)$ is a $\liek^{+}$-dominant vector we have,
\begin{equation*}
\dot X_{\alpha}\big(\dot X_{-1}^{\ell}(u)\big)=0 \quad
\text{for every}\quad \alpha\in\Delta^+(\liek,\lieh_\liek)-\{\gamma_1\}
\quad\text{and}\quad \ell \in \NN_0.
\end{equation*}
Then, in view of \eqref{Dk33}, it follows that
\begin{equation}\label{q+0}
\dot X\big(\dot X_{\delta}^{2i-\ell}\dot E^{j+\ell}(b_{2i,j})\big)=0
\quad \text{for every}\quad X \in \lieq^+.
\end{equation}
On the other hand, since $\dot E(\liey)=\dot X_4(\liey)=0$ and
$[\lieq^+,\liey]\subset \liey$ it follows that,
\begin{equation}\label{dotX(E)}
\dot X\big(E^n\big)\equiv 0 \quad \text{and} \quad \dot X\big(X_4^n\big)
\equiv 0\mod\!(U(\liek)\liey) \quad \text{for every} \quad X \in \lieq^+.
\end{equation}
Hence from \eqref{Dk}, \eqref{q+0} and \eqref{dotX(E)} we obtain that,
\begin{equation}\label{dotX(D)}
\dot X(D_k(b_{2i,j})) \equiv 0\mod\!(U(\liek)\liey) \quad
\text{for every}\quad X \in \lieq^+.
\end{equation}

Now, since $\dot X_1(E)=X_4$ and $\dot X_1(X_4)=0$,
using \eqref{Dk11} it follows that $\dot X_1(D_k(b_{2i,j}))=0$. The details
of this calculation can be found in the proof of Lemma 6.5 of \cite{BCT}.
Finally, it is easy to check that each term of $D_k(b_{2i,j})$ is a vector of weight $i(\gamma_4+\delta)+(j+k)\gamma_3$ with respect to $\lieh_\liek$.\qed
\end{proof}

\smallskip

As indicated at the beginning of the section we are interested in proving
that $P(T)$ implies $P(T-1)$ for $m\le T\le 2d_0$. To do this we need to show
that the $K$-types $b^r_{2i,T-r-2i}$ that occur in the $T-r$ skew diagonal
of $b_r$ are equal to zero for $0\le r\le m$. That is,
\begin{equation*}\label{P(T-1)} b^r_{2i,T-r-2i}=0 \quad\text
{if}\quad 0\le T-r-2i\le\min\big\{T,2d_0-T\big\}-r,
\end{equation*}
for $0\le r\le m$. For this purpose we introduce another
propositional function $Q(n)$ defined for $0\le
n\le\min\big\{T,2d_{0}-T\big\}+1$ as follows,
\begin{equation}\label{Q(n)}
Q(n):\,\, b^r_{2i,T-r-2i}=0 \quad \text {if}\quad 0\le T-r-2i<n
\quad \text {for}\quad 0\le r\le m.
\end{equation}
\noindent Clearly $Q(0)$ is true. Also, since we have that
$d(b_r)\le 2d_r$ for $0\le r\le m$, we obtain that \eqref{Q(n)} holds if
$T-r-2i>\min\big\{T,2d_{0}-T\big\}-r$.

\smallskip

Next theorem is the analogue of Theorem 6.6 of \cite{BCT} and its proof
is the same as that of Theorem 6.6, it consist in rewriting the sum
$\mathcal E_L^1(n)$ in terms of the vectors $D_k(b_{2i,j})$ defined in
Lemma \ref{LemDk}, and the sum $\mathcal E_L^2(n)$ in terms of
$\liek^+$-dominant vectors. We refer the reader to Section 6 of \cite{BCT}
for the details.

\begin{thm}\label{Thm 6} Let $b=b_m\otimes Z^m+\cdots+b_0\in B$
be such that $d(b_r)\le 2d_r$ for $0\le r\le m$, let $m\le T\le 2d_0$
and $0\le n\le\min\big\{T,2d_{0}-T\big\}$. Then if $P(T)$ and $Q(n)$
are true we have,
\begin{equation}\label{sist3}
\begin{split}
\sum_{\substack{ r,k\\T-L\le 2k+r\le T-n}}&
B_{r,k}(T,n,L)D_{L+2k+r-T}(b^r_{2k,T-r-2k})(X_{\delta}X_4)^{T-k}E^n\\
-\sum_{\substack{r,\ell\\r\equiv T-n}}&(-2)^\ell\textstyle\binom
L\ell\textstyle \binom{T-n-\ell}{r-\ell}
u^r_{T-r-n,n}(X_{\delta}X_4)^{(T+r+n)/2}E^L\equiv0,
\end{split}
\end{equation}
for all $L$ such that $0\leq L\leq \min\{2m,T\}-n$. Here the congruence is module
the left ideal $U(\liek)\liem^+$, $u^r_{T-r-n,n}=r!(-1)^r
\dot X_{\delta}^{T-n-r}\dot E^n(b^r_{T-n-r,n})$ and
\begin{equation*}\label{Brk}
B_{r,k}(T,n,L)=r!(-1)^T2^{T-r-2k}\textstyle\binom
L{T-r-2k}\textstyle \binom{T-L-n}{r-n}.
\end{equation*}
Moreover, the left hand side of equation \eqref{sist3} is a
weight vector of weight $T(\gamma_4+\delta)+(n+L)\gamma_3$.
\end{thm}

We are now in a good position to obtain the system of equations that we are
looking for. Using the notation introduced in \eqref{ST} define,
\begin{equation}\label{U}
U=X_{\delta}X_4-T_{23}S_{23}+T_{24}S_{24}.
\end{equation}
Then $U$ is a $\liek^+$-dominant vector of weight $\gamma_4+\delta$ with
respect to $\lieh_{\liek}$  and $U\equiv X_{\delta}X_4 \mod(U(\liek)\liey)$.
For any $T$ and $n$ such that $m\leq T\le 2d_0$ and $0\le n\le\min\{T,2d_{0}-T\}$
consider the following sets,
\begin{equation*}\label{LR}
\begin{split}
L(T,n)=&\left\{L\in{\NN_0}: 0\leq L\leq\min\{2m,T\}-n,\ \ L\not\equiv n\,\right\},\\
R_F(T,n)=&\left\{r\in{\NN_0}: 0\leq
r\leq\min\{m,\min\{T,2d_{0}-T\}-n\},\ r \equiv T-n\,\right\},
\end{split}
\end{equation*}
the congruence is\!\!\! $\mod(2)$ and the subindex $F$ stands for $F_4$.
Let $|L(T,n)|$ and $|R_F(T,n)|$ denote the cardinality of these sets. The set
$L(T,n)$ was also considered for the symplectic group Sp($n$,1) while
$R_F(T,n)$ is the analogue of the set $R(T,n)$ defined in Section 6 of \cite{BCT}.

\smallskip

Next theorem gives a system of linear equations where the unknowns, $u^r_{T-r-n,n}$,
are $\liek^+$-dominant vectors associated to the $K$-types that occur in the $T-r$
skew diagonal of the coefficient $b_r$ of $b$ for $0\le r\le m$ (see (\ref{array})).

\begin {thm}\label{sistLR} Let $b=b_m\otimes Z^m+\cdots+b_0\in B$ be such that
$d(b_r)\le 2d_r$ for $0\le r\le m$, and let $m\le T\le 2d_0$ and
$0\le n\le\min\big\{T,2d_{0}-T\big\}$. Then if $P(T)$  and $Q(n)$ are true
we have,
\begin{equation}\label{sistLR1}
\sum_{r\in R_F(T,n)}\Big(\sum_\ell(-2)^\ell\textstyle\binom
L\ell\textstyle \binom{T-n-\ell}{r-\ell}\Big)\, u^r_{T-r-n,n}\, U^{(T+r+n)/2}=0,
\end{equation}
for every $L\in L(T,n)$. Here $u^r_{T-r-n,n}=r!(-1)^r\dot X_{\delta}^{T-n-r}
\dot E^n(b^r_{T-n-r,n})$.
\end{thm}

\begin{proof} Let $u$ denote the left hand side of equation \eqref{sist3}. Then, in
view of Theorem \ref{Thm 6}, $u$ is a vector in $U(\liek)\liem^+$ of weight
$\lambda=T(\gamma_4+\delta)+(n+L)\gamma_3$ with respect to $\lieh_{\liek}$.

On the other hand, using that $\dot X\big(X_{\delta}\big)\equiv 0
\mod\!(U(\liek)\liey)$ for every $X \in \lieq^+$, together with \eqref{dotX(E)},
\eqref{dotX(D)} and the fact that $E$, $X_4$ and $X_{\delta}$ commute with
$\liey$ and that $[\lieq^+,\liey]\subset \liey$, it follows
that $\dot X(u)\equiv 0\mod\!(U(\liek)\liey)$ for every $X \in \lieq^+$. Then
applying Theorem \ref{m+/n} we obtain that $u\equiv0\mod\!(U(\liek)\liey)$, that is,
\begin{equation}\label{sist3/n}
\begin{split}
\sum_{\substack{ r,k\\T-L\le 2k+r\le T-n}}&
B_{r,k}(T,n,L)D_{L+2k+r-T}(b^r_{2k,T-r-2k})(X_{\delta}X_4)^{T-k}E^n\\
-\sum_{\substack{r,\ell\\r\equiv T-n}}&(-2)^\ell\textstyle\binom
L\ell\textstyle \binom{T-n-\ell}{r-\ell}
u^r_{T-r-n,n}(X_{\delta}X_4)^{(T+r+n)/2}E^L\equiv0.
\end{split}
\end{equation}

Since $U\equiv X_{\delta}X_4$\!$\mod(U(\liek)\liey)$ (see \eqref{U}), we replace
$X_{\delta}X_4$ by $U$ in \eqref{sist3/n}. Also, recall that $\dot X_1(X_{\delta})=
\dot X_1(X_4)=0$ and $\dot X_1(D_k(b_{2i,j}))=0$ for $b_{2i,j}\in U(\liek)^M$
of type $(2i,j)$ and $0\leq k\leq 2i$ (see Lemma \ref{LemDk}). Hence, since
$L\not\equiv n\mod(2)$, it follows from Proposition \ref{parimpar} and
Lemma \ref{U(g)l2} that
\begin{equation}\label{sistLR1/n}
\sum_{r\in R_F(T,n)}\Big(\sum_\ell(-2)^\ell\textstyle\binom
L\ell\textstyle \binom{T-n-\ell}{r-\ell}\Big)\, u^r_{T-r-n,n}
U^{(T+r+n)/2}\equiv 0,
\end{equation}
module the left ideal $U(\liek)\liey$.
Now, since the left hand side of equation \eqref {sistLR1/n} is a
$\liek^+$-dominant vector of weight $T(\gamma_4+\delta)+n \gamma_3$,
applying Theorem \ref{k+dom} we can replace the congruence mod($U(\liek)\liey$)
by an equality. This completes the proof of the theorem.\qed
\end{proof}

\smallskip

For $T$ and $n$ fixed, Theorem \ref{sistLR} gives a system of $|L(T,n)|$
linear equations in the $|R_F(T,n)|$ unknowns $u^r_{T-r-n,n}$. This system is
the analogue of the one given in Theorem 6.7 of \cite{BCT}. The main advantage
of this system is that the unknowns are all $\liek^+$-dominant vectors.
Let $A(T,n)$ denote the coefficient matrix of this system. In Section 6 of
\cite{BCT} a very thorough study of this matrix is carried out (see Subsection 6.2).
This is done by considering a $(k+1)\times(k+1)$ matrix $A(s)$ with polynomial
entries $A_{ij}(s)\in\CC[s]$ that generalizes $A(T,n)$. This matrix is defined as follows,
\begin{equation*}\label{Aij}
A_{ij}(s)=\sum_{0\le\ell\le\min\{L_i,2j+\delta\}}(-2)^\ell
\binom{L_i}\ell\binom{s-\ell}{2j+\delta-\ell},
\end{equation*}
where $0\le L_0<\cdots<L_k$ is a sequence of integers and $\delta\in\{0,1\}$. In
Theorem 6.15 of \cite{BCT} it is obtained an explicit formula  for $\det A(s)$
as a product of  polynomials of degree one in $s$. Hence, we know the exact values
of $s$ for which $A(s)$ is singular. Moreover, from the proof of Theorem 6.15
it follows that whenever $A(s)$ is singular the reason is that it has several pairs
of equal rows. In this case the strategy consist in replacing one equation
in each one of these pairs by a new equation obtained from Theorem \ref{Thm 6}.
We refer the reader to Subsection 6.3 of \cite{BCT} for the details.

\smallskip

Since our goal in this section is to prove Theorem \ref{Btilde2} we need to
restate Theorem \ref{sistLR} for elements $b\in\widetilde B$. If $b=
\sum_{r=0}^m b_r \otimes Z^r \in \widetilde B$, it follows from \eqref{Btilde}
that for $r$ even we have $b^{r}_{2i,j}=0$\, if\, $d(b^{r}_{2i,j})=2(i+j)\le r$.
Hence, when $T-n \equiv 0$ and $r\in R_F(T,n)$ is such that $d(b^r_{T-r-n,n})=T-r+n\le r$,
we have $u^r_{T-r-n,n}=0$ in equation \eqref{sistLR1}. Then we may consider a new
index set defined as follows,
\begin{equation}\label{Rtilde1}
\widetilde R_F(T,n)=\begin{cases}
\{r\in R_F(T,n):r<\tfrac{T+n}2 \},&\text{if $T-n\equiv 0$}\\
R_F(T,n),&\text{if $T-n\equiv 1$},
\end{cases}
\end{equation}
where the congruence is\!\!\! $\mod(2)$. For $b\in\widetilde B$ we restate
Theorem \ref{sistLR} as follows. This theorem is the analogue of Theorem 6.19
of \cite{BCT} and it will be our main tool in the proof of Theorem \ref{Btilde2}.

\begin {thm}\label{sistLR2} Let $b=b_m\otimes Z^m+\cdots+b_0\in\widetilde B$ be such
that $d(b_r)\le 2d_r$ for $0\le r\le m$, and let $m\le T\le 2d_0$ and
$0\le n\le\min\big\{T,2d_{0}-T\big\}$. Then if $P(T)$  and $Q(n)$ are true
we have,
\begin{equation*}\label{sistLR3}
\sum_{r\in\widetilde R_F(T,n)}\Big(\sum_\ell(-2)^\ell\textstyle\binom
L\ell\textstyle \binom{T-n-\ell}{r-\ell}\Big)\, u^r_{T-r-n,n}\, U^{(T+r+n)/2}=0,
\end{equation*}
for every $L\in L(T,n)$. Here $u^r_{T-r-n,n}=r!(-1)^r\dot X_{\delta}^{T-n-r}
\dot E^n(b^r_{T-n-r,n})$.
\end{thm}

\smallskip

Now we recall the definition of the sets $R(T,n)$ and $\widetilde R(T,n)$
used in the case of the group Sp($n$,1) (see Section 6 of \cite{BCT}).
Let $b=b_m\otimes Z^m+\cdots+b_0\in\widetilde B$ with $b_m\neq 0$. For positive
integers $T$ and $n$ such that $m\leq T\le4m$ and $0\le n\le\min\{T,4m-T\}$ consider
the following set
\begin{equation*}
R(T,n)=\left\{r\in{\NN_0}: 0\leq r\leq\min\{m,\min\{T,4m-T\}-n\},\,r \equiv T-n \right\},
\end{equation*}
where the congruence is \!\!\!$\mod(2)$. The set $\widetilde R(T,n)$ is defined as in
\eqref{Rtilde1} replacing $R_F(T,n)$ by $R(T,n)$ (see (116) in \cite{BCT}). Next we will
show that Theorem \ref{Btilde2} follows from Proposition 6.21 and Proposition 6.22 of
\cite{BCT}.

\medskip

\noindent{\bf Proof of Theorem \ref{Btilde2}.}
Let $b=b_m\otimes Z^m+\cdots+b_0\in \widetilde B$ be such that $d(b_r)\le 2d_r$
for $0\le r \le m$. We need to show that $b=0$. Assume on the contrary that
$b\neq 0$ and that $m=\deg(b)$, that is $b_m\neq0$. We will obtain a contradiction
by showing that $b_m=0$. In view of the definition of $\widetilde B$ (see \eqref{Btilde})
to do this it is enough to show that $P\left(\frac{3m}{2}\right)$ holds if $m$ is even and
that $P(m-1)$ is true if $m$ is odd. Since $P(2d_0)$ holds (see \eqref{types} and
\eqref{Pfunction}) this will follow from the fact that $P(T)$ implies
$P(T-1)$ for any $m\le T\le 2d_0$.

\smallskip

Consider first $m\ge 1$. Let $m\le T\le 2d_0$ and $0\le n\le\min\big\{T,2d_{0}-T\big\}$,
and assume that $P(T)$ and $Q(n)$ hold. Since $2d_0\le 4m$, it follows that $\min\big\{T,2d_{0}-T\big\}\le\min\{T,4m-T\}$ and a simple calculation shows that
$$\min\{m,\, \min\{T,2d_0-T\}-n\}\le \min\{m,\, \min\{T,4m-T\}-n\}.$$
Hence $R_F(T,n)\subset R(T,n)$ and therefore $\widetilde R_F(T,n)
\subset \widetilde R(T,n)$.

Now set, $u^r_{T-r-n,n}=0$ if $r\in \widetilde R(T,n)$ and $r\not\in \widetilde R_F(T,n)$
and $u^r_{T-r-n,n}=r!(-1)^r\dot X_{\delta}^{T-n-r}\dot E^n(b^r_{T-n-r,n})$
if $r\in \widetilde R_F(T,n)$. Then from Theorem \ref{sistLR2} we obtain for every
$L\in L(T,n)$ that,
\begin{equation}\label{sistLR4}
\sum_{r\in\widetilde R(T,n)}\Big(\sum_\ell(-2)^\ell\textstyle\binom
L\ell\textstyle \binom{T-n-\ell}{r-\ell}\Big)\, u^r_{T-r-n,n}\, U^{(T+r+n)/2}=0,
\end{equation}
Observe that, except for the fact that the vector $X_{\delta}X_4$ is replaced by $U$,
the system of equations given by \eqref{sistLR4} is the same as that of Theorem 6.19
of \cite{BCT}, in particular, their coefficient matrices are exactly the same. Then
that $P(T)$ implies $P(T-1)$ for any $m\le T\le 2d_0$ follows from Proposition 6.21
and Proposition 6.22 of \cite{BCT}. We point out that the proof of these propositions
are based on a very thorough study of the coefficient matrix of these system. We refer
the reader to Theorem 6.15, Corollary 6.16 and Proposition 6.20 of \cite{BCT} for the
details.

Consider now $m=0$. Assume that $b=b_0\in \widetilde B$, $b\neq 0$, and that
$d(b)=d(b_0)\le 2d_0=2$. From  the definition of $\widetilde B$ (see \eqref{Btilde})
we have  $b=b_0=b_{2,0}^0+ b_{0,1}^0$, therefore $b_{2,0}^0\neq 0$ or $b_{0,1}^0\neq 0$, in particular
$d(b)=2$. Consider the element $b^{2}\omega = b^{2} \otimes Z^2+\omega_1b^{2} \otimes Z+
b^{2}\omega_0 \in B$, where $\omega = 1\otimes Z^2+\omega_1\otimes Z+\omega_0$ is the
element in $P(U(\lieg))^K$ defined in Lemma \ref{casimir}.

From Proposition \ref{d(uv)} we have $d(b^2)=4$, hence the component of Kostant
degree four of $b^2$ is nonzero. Now, as in Proposition \ref{Btilde1}, we can remove the
components of Kostant degree less or equal to two from $b^{2}$ and the components of Kostant
degree less or equal to zero from $b^{2}\omega_0$. This procedure defines an element
$\widetilde b=\widetilde b_{2} \otimes Z^2+\widetilde b_{1}\otimes Z+\widetilde b_{0}
\in\widetilde B$ with $d(\widetilde b_{r})\le 2d_r$ for $0\le r\le 2$, and such that
the component of Kostant degree four of $\widetilde b_{2}$ is the same as that of $b^{2}$.
Then $\widetilde b\neq 0$, which contradicts the first part of the proof. Therefore $b=0$,
as we wanted to prove.\qed

\newcommand\bibit[5]{\bibitem
{#2}#3, {\em #4,\!\! } #5}

\end{document}